\documentclass[a4paper, leqno, 12pt]{article}

\title{\textsc{Arithmetic Riemann-Roch and Hilbert-Samuel formulae for pointed stable curves}}
\author{\textsc{Gerard Freixas i Montplet}}
\date{}

\usepackage{amsmath, mathrsfs, amssymb, amscd, amsthm}
\usepackage[]{fontenc}
\usepackage{xy}
\usepackage{enumerate}
\xyoption{all}

\numberwithin{equation}{section}

\theoremstyle{plain}
\newtheorem{theorem}{Theorem}[section]
\newtheorem{proposition}[theorem]{Proposition}
\newtheorem{lemma}[theorem]{Lemma}
\newtheorem{corollary}[theorem]{Corollary}
\newtheorem*{theoremA}{Theorem A}
\newtheorem*{theoremB}{Theorem B}

\theoremstyle{definition}
\newtheorem{definition}[theorem]{Definition}
\newtheorem{construction}[theorem]{Construction}

\theoremstyle{remark}
\newtheorem{remark}[theorem]{Remark}

\DeclareMathOperator{\Spec}{Spec}
\DeclareMathOperator{\hyp}{hyp}
\DeclareMathOperator{\pre}{pre}
\DeclareMathOperator{\ACH}{\widehat{CH}}
\DeclareMathOperator{\CH}{CH}
\DeclareMathOperator{\ac1}{\widehat{c}_{1}}
\DeclareMathOperator{\c1}{c_{1}}
\DeclareMathOperator{\PSL}{PSL}

\DeclareMathOperator{\an}{an}
\DeclareMathOperator{\Real}{Re}
\DeclareMathOperator{\Imag}{Im}
\DeclareMathOperator{\adeg}{\widehat{deg}}

\DeclareMathOperator{\Res}{Res}

\DeclareMathOperator{\Pet}{Pet}
\DeclareMathOperator{\Div}{Div}

\newcommand{\OO}{\mathcal{O}}
\newcommand{\XX}{\mathcal{X}}
\newcommand{\BS}{\mathcal{S}}
\newcommand{\UU}{\mathcal{U}}
\newcommand{\CC}{\mathbb{C}}
\newcommand{\Int}{\mathbb{Z}}
\newcommand{\SCM}{\mathcal{\overline{M}}}
\newcommand{\PP}{\mathbb{P}}
\newcommand{\HH}{\mathbb{H}}
\newcommand{\RR}{\mathbb{R}}
\newcommand{\QQ}{\mathbb{Q}}
\newcommand{\FF}{\mathbb{F}}

\newcommand{\pd}{\partial}
\newcommand{\cz}{\overline{z}}
\newcommand{\cu}{\overline{u}}

\newcommand{\C}{\mathcal{C}}
\newcommand{\SCC}{\mathcal{\overline{C}}}
\newcommand{\SM}{\mathcal{M}}
\newcommand{\DD}{\mathcal{D}}
\newcommand{\ff}{\mathfrak{f}}
\newcommand{\fg}{\mathfrak{g}}

\newcommand{\fY}{\mathfrak{Y}}
\newcommand{\fZ}{\mathfrak{Z}}

\newcommand{\cpd}{\overline{\pd}}

\begin{document}
\setcounter{tocdepth}{1}
\setcounter{section}{0}
\maketitle

\begin{abstract} 
Let $(\OO,\Sigma, F_{\infty})$ be an arithmetic ring of Krull dimension at most 1, $\BS=\Spec\OO$ and $(\XX\rightarrow\BS;\sigma_{1},\ldots,\sigma_{n})$ a pointed stable curve. Write $\UU=\XX\setminus\cup_{j}\sigma_{j}(\BS)$. For every integer $k\geq 0$, the invertible sheaf $\omega_{\XX/\BS}^{k+1}(k\sigma_{1}+\ldots+k\sigma_{n})$ inherits a singular hermitian structure from the hyperbolic metric on the Riemann surface $\UU_{\infty}$. In this article we define a Quillen type metric $\|\cdot\|_{Q}$ on the determinant line $\lambda_{k+1}=\lambda(\omega_{\XX/\BS}^{k+1}(k\sigma_{1}+\ldots+k\sigma_{n}))$ and we compute the arithmetic degree of $(\lambda_{k+1},\|\cdot\|_{Q})$ by means of an analogue of the Riemann-Roch theorem in Arakelov geometry. As a byproduct, we obtain an arithmetic Hilbert-Samuel formula: the arithmetic degree of $(\lambda_{k+1},\|\cdot\|_{L^{2}})$ admits an asymptotic expansion in $k$, whose leading coefficient is given by the arithmetic self-intersection of $(\omega_{\XX/\BS}(\sigma_{1}+\ldots+\sigma_{n}),\|\cdot\|_{\hyp})$. Here $\|\cdot\|_{L^{2}}$ and $\|\cdot\|_{\hyp}$ denote the $L^{2}$ metric and the dual of the hyperbolic metric, respectively.
\end{abstract}

\tableofcontents

\section{Introduction}
Let $(\OO, \Sigma, F_{\infty})$ be an arithmetic ring of Krull dimension at most $1$ \cite[Def. 3.1.1]{GS}. Set $\BS=\Spec\OO$ and let $\eta$ stand for its generic point. Let $(\pi:\XX\rightarrow\BS;\sigma_{1},\ldots,\sigma_{n})$ be a $n$-pointed stable curve of genus $g$, in the sense of Knudsen and Mumford \cite[Def. 1.1]{Knudsen}. Assume that $\XX_{\eta}$ is smooth. Write $\UU=\XX\setminus\cup_{j}\sigma_{j}(\BS)$. The connected components of the complex analytic space $\UU_{\infty}:=\sqcup_{\sigma\in\Sigma}\UU_{\sigma}(\CC)$ are hyperbolic Riemann surfaces of finite type. The stability hypothesis guarantees the existence of a unique complete hyperbolic metric of constant curvature $-1$ on $\UU_{\infty}$. This metric induces an arakelovian --i.e. invariant under the action of $F_{\infty}$-- hermitian structure $\|\cdot\|_{\hyp}$ on the invertible sheaf $\omega_{\XX/\BS}(\sigma_{1}+\ldots+\sigma_{n})$. Although $\|\cdot\|_{\hyp}$ is not smooth, its singularities are of some logarithmitc type: it is a \textit{pre-log-log} hermitian metric in the sense of Burgos-Kramer-K\"uhn \cite[Sec. 7]{BKK}. The generalizations of the arithmetic intersection theory by Bost \cite{Bost}, K\"uhn \cite{Kuhn} and Burgos-Kramer-K\"uhn \cite{BKK} allow to attach several arithmetic invariants to the metrized invertible sheaf $\omega_{\XX/\BS}(\sigma_{1}+\ldots+\sigma_{n})_{\hyp}$. For instance, if $\XX$ is regular, there is a first arithmetic Chern class $\ac1(\omega_{\XX/\BS}(\sigma_{1}+\ldots+\sigma_{n})_{\hyp})$ in a pre-log-log arithmetic Chow group $\ACH^{1}_{\pre}(\XX)$ \cite[Sec. 7]{BKK}. In general, and if $\OO$ is the ring of integers of a number field, there is an arithmetic self-intersection number $(\omega_{\XX/\BS}(\sigma_{1}+\ldots+\sigma_{n})_{\hyp})^{2}\in\RR$.

In this article we pursue our investigations on an arithmetic Riemann-Roch theorem for pointed stable curves, initiated in \cite{SingARR}. For the sake of simplicity, assume $\XX$ regular. In loc. cit. we exhibited a relation between $\ac1(\omega_{\XX/\BS}(\sigma_{1}+\ldots+\sigma_{n})_{\hyp})^{2}\in\ACH^{2}_{\pre}(\XX)$ and $\lambda(\omega_{\XX/\BS})_{Q}=(\lambda(\omega_{\XX/\BS}),\|\cdot\|_{Q})$. Here $\|\cdot\|_{Q}$ is a suitable Quillen type metric on $\lambda(\omega_{\XX/\BS}):=\det R\pi_{\ast}\omega_{\XX/\BS}$. Its definition involves the special values at $1$ of the derivatives of the Selberg zeta functions $Z(\UU_{\sigma}(\CC),s)$, $\sigma\in\Sigma$. A special feature of our formula is the appearance of an arithmetic counterpart $\psi_{W}$ of the so called \textit{psi} class on $\SCM_{g,n}$. This class detects the continuous spectrum in the resolution of the hyperbolic laplacian on $\UU_{\infty}$ --provided that $n>0$--. Its underlying hermitian structure is Wolpert's renormalization of the hyperbolic metric at the cusps (\cite[Def.1]{Wolpert:cusps} and Definition \ref{definition:Wolpert} below). In the present work we introduce a Quillen type metric on the determinant $\lambda(\omega_{\XX/\BS}^{k+1}(k\sigma_{1}+\ldots+k\sigma_{n}))$, for every integer $k\geq 0$ (Definition \ref{definition:Quillen}). If $k\geq 1$, this metric is built up with the special values $Z(\UU_{\sigma}(\CC), k+1)$, $\sigma\in\Sigma$, and the natural $L^{2}$ pairings on $H^{0}(\XX, \omega_{\XX/\BS}^{k+1}(k\sigma_{1}+\ldots+k\sigma_{n}))\otimes_{\sigma}\CC$, $\sigma\in\Sigma$. We relate $\lambda(\omega_{\XX/\BS}^{k+1}(k\sigma_{1}+\ldots+k\sigma_{n}))_{Q}$ and $\ac1(\omega_{\XX/\BS}(\sigma_{1}+\ldots+\sigma_{n}))^{2}$ through an arithmetic Riemann-Roch formula.
\begin{theoremA}
Let $(\OO,\Sigma,F_{\infty})$ be an arithmetic ring of Krull dimension at most $1$. Let $(\pi:\XX\rightarrow\BS;\sigma_{1},\ldots,\sigma_{n})$ be a pointed stable curve of genus $g$. For every closed point $\wp\in\BS$ denote by $n_{\wp}$ the number of singular points in the geometric fiber $\XX_{\wp}$ and put $\Delta_{\XX/\BS}=\left[\sum_{\wp}n_{\wp}\wp\right]\in\CH^{1}(\BS)$.\\
i. If $\XX$ is regular, then the identity
\begin{displaymath}
	\begin{split}
	12\ac1(\lambda(\omega_{\XX/\BS}^{k+1}(k\sigma_{1}+\ldots+&k\sigma_{n}))_{Q})-\Delta_{\XX/\BS}+\ac1(\psi_{W})=\\
	(&6k^{2}+6k+1)\pi_{\ast}\left(\ac1(\omega_{\XX/\BS}(\sigma_1+\ldots+\sigma_n)_{\hyp})^{2}\right)\\
	&+\ac1\left(\OO(C(g,n))\right)
	\end{split}
\end{displaymath}
holds in the arithmetic Chow group $\ACH^{1}(\BS)$.\\
ii. If $\OO$ is the ring of integers of a number field and $\XX$ is generically smooth, then there is an equality of real numbers
\begin{displaymath}
	\begin{split}
	12\adeg\ac1(\lambda(\omega_{\XX/\BS}^{k+1}(k\sigma_{1}+\ldots+&k\sigma_{n}))_{Q})-\adeg\Delta_{\XX/\BS}
	+\adeg\ac1(\psi_{W})=\\
	(&6k^{2}+6k+1)(\omega_{\XX/\BS}(\sigma_{1}+\ldots+\sigma_{n})_{\hyp})^{2}\\
	&+\adeg\ac1\left(\OO(C(g,n))\right).
	\end{split}
\end{displaymath}

\end{theoremA}
Notice that despite of being rather restrictive, our hypothesis are beyond the reach of the arithmetic Riemann-Roch theorem of Gillet-Soul\'e \cite{ARR}.

The proof of Theorem A relies on a more general statement: a metrized version of a Mumford type isomorphism on $\SM_{g,n}$ (Theorem \ref{theorem:5.1}). The techniques combine the geometry of the boundary stack $\pd\SM_{g+n,0}$ --through Knudsen's clutching morphisms-- and the behavior of the Selberg zeta function (Theorem \ref{theorem:4.2}) and the $L^{2}$ metric (Theorem \ref{theorem:4.3}) in suitable degenerating families of Riemann surfaces. Wolpert's pinching expansion for the family hyperbolic metric \cite[Exp. 4.2]{Wolpert:hyperbolic} proves to be a key result in the necessary explicit computations. As a remarkable outcome, Theorem \ref{theorem:5.1} implies the Takhtajan-Zograf local index formula for pointed stable curves (\cite{ZT1}--\cite{ZT2} and Corollary \ref{corollary:5.2} below).

It is worth pointing out that, in Theorem A, the stability condition at the finite places can be relaxed.  Provided that $g\geq 1$, one may assume that $\XX$ is generically smooth, semi-stable and $\sigma_{1},\ldots,\sigma_{n}:\BS\rightarrow\XX$ are disjoint sections in the smooth locus of $\pi$. This covers important examples such as the Deligne-Rapoport models of modular curves $X_{0}(p)\rightarrow\Spec\Int$, $p$ prime for which $g(X_{0}(p))\geq 1$, together with the sections $0$ and $\infty$ \cite{DR}. The proof combines an anomaly formula --established by means of Theorem \ref{theorem:5.1}-- and the arithmetic Riemann-Roch theorem of Gillet-Soul\'e in its more general form \cite{ARR}. The details will be worked out elsewhere.

In arithmetic applications a weaker form of Theorem A may be sometimes enough: an arithmetic Hilbert-Samuel type formula. As an advantage, it provides a geometric interpretation of the arithmetic self-intersection number $(\omega_{\XX/\BS}(\sigma_{1}+\ldots+\sigma_{n})_{\hyp})^{2}$.
\begin{theoremB}
Let $K$ be a number field, $\OO_{K}$ its ring of integers and $\BS=\Spec\OO_{K}$. Let $(\pi:\XX\rightarrow\BS;\sigma_{1},\ldots,\sigma_{n})$ be a pointed stable curve with generic fiber $\XX_{K}$ smooth. Then there is an asymptotic expansion
\begin{displaymath}
	\adeg\ac1(\lambda(\omega_{\XX/\BS}^{k+1}(k\sigma_{1}+\ldots+k\sigma_{n}))_{L^{2}})=
	\frac{k^{2}}{2}(\omega_{\XX/\BS}(\sigma_{1}+\ldots+\sigma_{n})_{\hyp})^{2}+\OO(k\log k).
\end{displaymath}
In particular, the following limit formula holds:
\begin{displaymath}
	\lim_{k\to +\infty}\frac{2}{k^{2}}\adeg\ac1(\lambda(\omega_{\XX/\BS}^{k+1}(k\sigma_{1}+\ldots+k\sigma_{n}))_{L^{2}})=
	(\omega_{\XX/\BS}(\sigma_{1}+\ldots+\sigma_{n})_{\hyp})^{2}.
\end{displaymath}
\end{theoremB}
Notice that the statement of Theorem B does not involve the Quillen metrics, but only the $L^{2}$ metrics on the determinant lines. By means of the uniformization theorem, these norms can be constructed from the Petersson metrics on suitable spaces of automorphic cusp forms. Taking for granted the extension of Theorem A to more general hypothesis, as alluded above, Theorem B becomes specially meaningful for models of modular curves $X_{0}(p)\rightarrow\Spec\Int$, $p$ prime with $g(X_{0}(p))\geq 1$, together with the sections $0$, $\infty$.

Possible applications of Theorem A concern the special values of the Selberg and Ruelle z\^eta functions for modular curves. A particular instance was already treated in \cite[Th. B]{SingARR}. Analogous but new results will be compiled in a future work. Hereby we use Theorem B to understand the first successive minimum of the hermitian bundles $(H^{0}(\XX,\omega_{\XX/\BS}^{k+1}(k\sigma_{1}+\ldots+k\sigma_{n})),\|\cdot\|_{\Pet})$, $k\geq 0$, for a pointed stable curve of genus $0$ and $\|\cdot\|_{\Pet}$ the Petersson norm (Section \ref{section:applications}).

Other approaches to Theorem A already exist in the literature. We cite for instance \cite[Part II]{Weng}. The method of loc. cit. leads to an analogous statement up to an unknown universal constant. The advantage of our formula is that explicit computations are allowed. Furthermore, Theorem B can \textit{not} be recovered from \cite{Weng}. Finally, in contrast with \cite[Fund. rel. IV', p. 280]{Weng}, our result is available for pointed stable curves of any genus.\footnote{The proof of loc. cit. presents a gap in genus $g\leq 2$ (last two lines in page 279). The case $g=2$ requires a justification whereas there are counterexamples to the principle of the proof in genus $0$ and $1$: there exist non-constant harmonic functions on $\SM_{g,n;\CC}$, for $g=0$, $n\geq 4$ and $g=1$, $n\geq 1$.  Also the case $g=0$ and $n=3$ is beyond the reach in \cite{Weng}.} We also refer to Hahn's forthcoming thesis \cite{Hahn}. Hahn obtains different results related to Theorem A. His approach is much in the spirit of Jorgenson--Lundelius \cite{JL1}--\cite{JL3}. In contrast with our geometric considerations, Hahn works with a degenerating family of metrics on a fixed compact Riemann surface and studies the behavior of the corresponding family of heat kernels. Consequences are derived for the family of heat trace regularizations and spectral zeta functions.

We briefly review the structure of this paper. In Section \ref{section:conventions} we compile most of the normalizations and notations that hold throughout the article. In Section \ref{section:clutching} we study the interplay between the adjunction isomorphism and Knudsen's clutching operation. The results obtained are applied in Section \ref{section:tautological}, where we recall the definition of the tautological invertible sheaves on $\SCM_{g,n}$ and prove relations between them. Section \ref{section:quillen} is devoted to the degeneracy of the family Quillen metric. Preliminaries on the continuity properties of the Selberg zeta function in family are provided. These are combined with a careful study of the degeneracy of the family $L^{2}$ metric, whose only merit is due to Obitsu and Wolpert. Theorem A, Theorem B and the Takhtajan-Zograf local index formula are proven in Section \ref{section:metrized}. Our results are exemplified through the case of pointed stable curves of genus 0, in Section \ref{section:applications}. Finally, in appendix we give a proof of the ampleness of the line bundle $\kappa_{0,n}$ on $\SCM_{0,n}\rightarrow\Spec\Int$, for the sake of completeness. 

\textbf{Acknowledgements}
It is a pleasure to thank J.-B. Bost and J. I. Burgos, who proposed me to study the topic of this article during my PhD. thesis, as well as a possible approach: it revealed to be fruitful. S. Wolpert kindly guided me at some points and shared his knowledge in Teichm\"uller theory. He also communicated me the reference \cite{Obitsu-Wolpert}, crucial for this work. I am deeply indebted to him. I acknowledge J.-M. Bismut for stimulating discussions on questions related to this paper. Finally I am grateful to T. Hahn, J. Kramer, U. K\"uhn, for the mathematical exchanges we had throughout the last years. This work was conceived while I was visiting the Universit\`a degli Studi di Padova, with the support of the European Commission's \textquotedblleft Marie Curie early stage researcher\textquotedblright\, programme (6th Framework Programme RTN \textit{Arithmetic Algebraic Geometry}). I thank these institutions for their hospitality and financial support.

\section{Conventions and notations}\label{section:conventions}
We fix some conventions and notations that will hold throughout this paper.

\paragraph{Constants.} Let $g,n\geq 0$ be integers with $k\geq 0$, $2g-2+n>0$. We define the real constants
\begin{displaymath}
	\begin{split}
		&C(g,n)=\exp\left((2g-2+n)\left(\frac{\zeta^{\prime}(-1)}{\zeta(-1)}+\frac{1}{2}\right)\right),\\
		&E_{1}(g,n)=2^{(g+2-n)/3}\pi^{-n/2}\\
		&\hspace{1.9cm}\cdot\exp\left((2g-2+n)\left(2\zeta^{\prime}(-1)-\frac{1}{4}+\frac{1}{2}\log(2\pi)\right)\right)
	\end{split}
\end{displaymath}
and, for every integer $k\geq 1$,
\begin{displaymath}
	\begin{split}
		&E_{k+1}(g,n)=2^{((3k+1)(g-1)-n)/3}((2k)!\pi^{2k+1})^{-n/2}\\
		&\hspace{2.3cm}\cdot\exp\Big((2g-2+n)\big(2\zeta^{\prime}(-1)-(k+\frac{1}{2})^{2}+(k+\frac{1}{2})\log(2\pi)\\
		&\hspace{3.5cm}+\sum_{j=1}^{2k}(j-k-\frac{1}{2})\log j\big)\Big).
	\end{split}
\end{displaymath}
Here $\zeta$ denotes the Riemann zeta function. Notice the relations
\begin{align}
	&C(g+n,0)=C(g,n)C(1,1)^{n},\label{equation:C}\\
	&E_{1}(g+n,0)=\pi^{n}E_{1}(g,n)E_{1}(1,1)^{n},\label{equation:E_{1}}\\
	&E_{k+1}(g+n,0)=(2^{k+1}(2k)!\pi^{2k+1})^{n}E_{k+1}(g,n)E_{k+1}(1,1)^{n},\quad k\geq 1.
\end{align}

\paragraph{Hyperbolic and Wolpert metrics.} Let $X$ be a compact and connected Riemann surface of genus $g$ and $p_{1}$,$\ldots$, $p_{n}$ distinct points in $X$, $2g-2+n>0$. The open subset $U=X\setminus\lbrace p_{1},\ldots,p_{n}\rbrace$ admits a complete hyperbolic riemannian metric, of constant curvature -1. Denote it by $ds_{\hyp,U}^{2}$. Via a fuchsian uniformization $U\simeq\Gamma\backslash\HH$, $\Gamma\subset\PSL_{2}(\RR)$ torsion free, the metric $ds_{\hyp,U}^{2}$ is obtained by descent from the $\Gamma$ invariant Riemann tensor on $\HH$
\begin{displaymath}
	ds_{\hyp,\HH}^{2}=\frac{dx^{2}+dy^{2}}{y^{2}},\,\,\,z=x+iy\in\HH.
\end{displaymath}
Associated to $ds_{\hyp,U}^{2}$ there is a hermitian metric on the complex line $T_{U}$, that we write $h_{U}$. It is obtained by descent from the metric $h_{\HH}$ on $T_{\HH}$ defined by the rule
\begin{displaymath}
	h_{\HH}\left(\frac{\pd}{\pd z},\frac{\pd}{\pd z}\right)=\frac{1}{2y^{2}}.
\end{displaymath}
The hermitian metric $h_{U}$ extends to a pre-log-log hermitian metric $\|\cdot\|_{\hyp}$ on $\omega_{X}(p_{1}+\ldots+p_{n})$ \cite[Sec. 7.3.2]{GFM}. The first Chern form of $\omega_{X}(p_{1}+\ldots+p_{n})_{\hyp}$, which is defined on $U$, coincides with the normalized K\"ahler form $\omega$ of $h_{U}$ (curvature $-1$ condition). The form $\omega$ is locally given by
\begin{displaymath}
	\omega=\frac{i}{2\pi}h_{U}\left(\frac{\pd}{\pd z},\frac{\pd}{\pd z}\right) dz\wedge d\cz.
\end{displaymath}
The volume of $X$ with respect to $\omega$ is $2g-2+n$.

For every puncture $p_{j}$ there is a conformal coordinate $z$ with $z(p_{j})=0$, by means of which a small punctured disc $D^{*}(0,\varepsilon)\subset\CC$ with the Poincar\'e metric
\begin{displaymath}
	ds_{P}^{2}=\left(\frac{|dz|}{|z|\log|z|}\right)^{2}
\end{displaymath}
isometrically embeds into $(U,ds_{\hyp,U}^{2})$. Such a coordinate is unique up to rotation and is called a \textit{rs} coordinate at the cusp $p_{j}$.
\begin{definition}[Wolpert metric \cite{Wolpert:cusps}, Def. 1]\label{definition:Wolpert}
Let $z$ be a \textit{rs} coordinate at the cusp $p_{j}$. The \textit{Wolpert metric} on the complex line $\omega_{X,p_{j}}$ is defined by
\begin{displaymath}
	\|dz\|_{W,p_{j}}=1.
\end{displaymath}
The tensor product $\otimes_{j}\omega_{X,p_{j}}$ is equipped with the tensor product of Wolpert metrics, and we write $\|\cdot\|_{W}$ for the resulting metric.
\end{definition}
\paragraph{$L^{2}$ metrics.} Let $(X;p_{1},\ldots,p_{n})$ be as above. We introduce the $L^{2}$ metrics on the determinant lines $\lambda(\omega_{X}^{k+1}(kp_{1}+\ldots+kp_{n}))$. The cases $k=0$ and $k\geq 1$ need to be distinguished.

The complex vector space $\C^{\infty}(X,\omega_{X})(\supset H^{0}(X,\omega_{X}))$ is equipped with the non-degenerate hermitian form
\begin{displaymath}
	\langle\alpha,\beta\rangle_{0}=\frac{i}{2\pi}\int_{X}\alpha\wedge\overline{\beta}.
\end{displaymath}
The space $H^{1}(X,\omega_{X})^{\vee}$ is canonically isomorphic to $H^{0}(X,\OO_{X})$ $=\CC$ via the analytic Serre duality. Since $\omega$ is integrable, the $L^{2}$ metric on $\C^{\infty}(X,\OO_{X})(\supset H^{0}(X,\OO_{X}))$ with respect to $h_{U}$ is well defined. If $\mathbf{1}$ is the function with constant value 1, then
\begin{displaymath}
	\langle\mathbf{1},\mathbf{1}\rangle_{1}=\int_{X}\omega=2g-2+n.
\end{displaymath}
The complex line $\lambda(\omega_{X})=\det H^{0}(X,\omega_{X})\otimes\det H^{1}(X,\omega_{X})^{-1}$ is endowed with the determinant metric build up from $\langle\cdot,\cdot\rangle_{0}$ and $\langle\cdot,\cdot\rangle_{1}$. We refer to it by $\|\cdot\|_{L^{2}}$.

Let $k\geq 1$ be an integer. Notice that $H^{1}(X,\omega_{X}^{k+1}(kp_{1}+\ldots+kp_{n}))=0$. Hence we have 
\begin{displaymath}
	\begin{split}
	\lambda(\omega_{X}^{k+1}(kp_{1}+\ldots+kp_{n})):=&\det R\Gamma(X,\omega_{X}^{k+1}(kp_{1}+\ldots+kp_{n}))\\
	=&\det H^{0}(X,\omega_{X}^{k+1}(kp_{1}+\ldots+kp_{n})).
	\end{split}
\end{displaymath}
The line bundle $\omega_{X}^{k+1}(kp_{1}+\ldots+kp_{n})$ inherits a --singular-- hermitian structure from $\omega_{X}(p_{1}+\ldots+p_{n})_{\hyp}^{k+1}$. Denote it by $\langle\cdot,\cdot\rangle_{\hyp}$. For every $\alpha,\beta\in H^{0}(X,\omega_{X}^{k+1}(kp_{1}+\ldots+kp_{n}))$, the integral
\begin{displaymath}
	\langle\alpha,\beta\rangle_{L^{2}}:=\int_{X}\langle\alpha,\beta\rangle_{\hyp}\omega
\end{displaymath}
is seen to be absolutely convergent.\footnote{However this fails to be true for the bigger space $H^{0}(X,\omega_{X}(\sigma_{1}+\ldots+\sigma_{n})^{k+1})$.} This rule defines a hermitian metric on $H^{0}(X,\omega_{X}^{k+1}(kp_{1}+\ldots+kp_{n}))$. We write $\|\cdot\|_{L^{2}}$ for the determinant metric on $\lambda(\omega_{X}^{k+1}(kp_{1}+\ldots+kp_{n}))$.

\paragraph{Selberg zeta function and Quillen metric.} We next recall the definition of the Selberg zeta function of $U$ (see \cite{Hejhal}). For every real $l>0$ the function
\begin{displaymath}
	Z_{l}(s)=\prod_{k=0}^{\infty}(1-e^{-(s+k)l})^{2}
\end{displaymath}
is holomorphic in $\Real s>0$. In a first step, the Selberg zeta function of $U$ is defined by the absolutely convergent product
\begin{displaymath}
	Z(U,s)=\prod_{\gamma}Z_{l(\gamma)}(s),\,\,\,\Real s>1,
\end{displaymath}
running over the simple closed non-oriented geodesics of the hyperbolic surface $(U,ds_{\hyp,U}^{2})$. Then one shows that $Z(U,s)$ extends to a meromorphic function on $\CC$, with a simple zero at $s=1$.
\begin{definition}[Quillen metric]\label{definition:Quillen}
Let $k\geq 0$ be an integer. We define the \textit{Quillen metric} on $\lambda(\omega_{X}^{k+1}(kp_{1}+\ldots+kp_{n}))$, attached to the hyperbolic metric on $U$, to be
\begin{displaymath}
	\|\cdot\|_{Q}=\begin{cases}
		(E_{1}(g,n)Z^{\prime}(U,1))^{-1/2}\|\cdot\|_{L^{2}},	&\text{if}\quad k=0,\\
		(E_{k+1}(g,n)Z(U,k+1))^{-1/2}\|\cdot\|_{L^{2}},	&\text{if}\quad k\geq 1.
		\end{cases}
\end{displaymath}
\end{definition}
Notice the apparent discrepancy between the definition of the Quillen metric for $k=0$ and for $k\geq 1$ (recall the definition of $E_{1}(g.n)$ and $E_{k+1}(g,n)$ for $k\geq 1$). The difference stems from the fact $\dim H^{0}(X,\omega_{X})=\chi(\omega)+1$, while $\dim H^{0}(X,\omega_{X}^{k+1}(kp_{1}+\ldots+kp_{n}))=\chi(\omega_{X}^{k+1}(kp_{1}+\ldots+kp_{n}))$ for $k\geq 1$.

\paragraph{Arithmetic setting.} Let $(\OO, \Sigma, F_{\infty})$ be an arithmetic ring of Krull dimension at most $1$. Put $\BS=\Spec\OO$ and denote its generic point by $\eta$. Let $(\pi:\XX\rightarrow\BS;\sigma_{1},\ldots,\sigma_{n})$ be a $n$-pointed stable curve of genus $g$, in the sense of Knudsen and Mumford \cite[Def. 1.1]{Knudsen}. Assume $\XX_{\eta}$ smooth. We write $\UU=\XX\setminus\cup_{j}\sigma_{j}(\BS)$. By definition $\XX_{\eta}$ is geometrically connected. For every complex embedding $\tau\in\Sigma$, the preceding constructions apply to $\UU_{\tau}(\CC)\subset\XX_{\tau}(\CC)$. Varying $\tau$, we obtain arakelovian hermitian line bundles $\omega_{\XX/\BS}(\sigma_{1}+\ldots+\sigma_{n})_{\hyp}$, $\psi_{W}=\otimes_{j}\sigma_{j}^{\ast}(\omega_{\XX/\BS})_{W}$, $\lambda(\omega_{\XX/\BS}^{k+1}(k\sigma_{1}+\ldots+k\sigma_{n}))_{Q}$. 

Furthermore we denote by $\OO(C(g,n))$ the trivial line bundle equipped with the norm $C(g,n)|\cdot|$. Here $|\cdot|$ denotes the absolute value at all archimedian places.

\paragraph{Miscellanea.} Similar notations will be employed for analogous constructions over more general bases. In this way, we shall consider the \textquotedblleft universal situation\textquotedblright\, $(\pi:\SCC_{g,n}\rightarrow\SCM_{g,n};\sigma_{1},\ldots,\sigma_{n})$, where $\SCM_{g,n}$ is the Deligne-Mumford stack of $n$-pointed stable curves of genus $g$, $\SCC_{g,n}$ the universal family and $\sigma_{1},\ldots,\sigma_{n}$ the universal sections. We then have \textquotedblleft universal hermitian line bundles\textquotedblright\, $\lambda(\omega_{\SCC_{g,n}/\SCM_{g,n}}^{k+1}(k\sigma_{1}+\ldots+k\sigma_{n}))\mid_{\SM_{g,n};Q}$, $\sigma^{\ast}_{j}(\omega_{\SCC_{g,n}/\SCM_{g,n}})\mid_{\SM_{g,n};W}$, etc. ($\SM_{g,n}$ is the open substack of smooth curves). When the context is clear enough, we freely write $\lambda_{k+1;g,n;Q}$, $\sigma_{j}^{\ast}(\omega_{\SCC_{g,n}/\SCM_{g,n}})_{W}$, $\psi_{g,n;W}$, etc.

If $F$ is an algebraic stack of finite type over $\Spec\Int$, then we denote by $F^{\an}$  the analytic stack associated to $F_{\CC}$. For instance, applied to $\SM_{g,n}$ and $\SCM_{g,n}$, we obtain the analytic stack $\SM_{g,n}^{\an}$ of $n$-punctured Riemann surfaces of genus $g$ and its Deligne-Mumford stable compactification $\SCM_{g,n}^{\an}$. To a morphism $F\rightarrow G$ between algebraic stacks of finite type over $\Spec\Int$, corresponds a morphism between analytic stacks $F^{\an}\rightarrow G^{\an}$.

For the theory of algebraic stacks we follow the references \cite{DM} and \cite{LMB}. For pointed stable curves we refer to the original article of Knudsen \cite{Knudsen}. Concerning the arithmetic intersection theory, we follow the several extensions of the theory of Gillet-Soul\'e \cite{GS}--\cite{GS2} developed by Bost \cite{Bost}, K\"uhn \cite{Kuhn} and Burgos-Kramer-K\"uhn \cite{BKK}--\cite{BKK2}. 
\section{Adjunction and clutching morphisms}\label{section:clutching}
\subsection{Adjunction isomorphisms for pointed stable curves}
Let $S$ be a noetherian scheme and $(\pi:X\rightarrow S; \sigma_{1},\ldots,\sigma_{n})$ a $n$-pointed stable curve of genus $g$ \cite[Def. 1.1]{Knudsen}, with $n\geq 1$. The sections $\sigma_{1},\ldots, \sigma_{n}$ lie in the smooth locus of $\pi$ and define disjoint relative effective Cartier divisors. We indistinctly write $\sigma_{1},\ldots, \sigma_{n}$ for the sections or for the Cartier divisors they define.

The theory of determinants of Knudsen-Mumford \cite[Ch. I]{Knudsen-Mumford} is applied below. We deal with \textit{graded} invertible $\OO_{S}$-modules of the form $(L,a)$, $a\in\Int$ (see loc. cit., p. 20). For two such graded invertible sheaves $(L,a)$, $(M,b)$, the commutativity law $(L,a)\otimes_{\OO_{S}} (M,b)\overset{\sim}{\rightarrow} (M,b)\otimes_{\OO_{S}} (L,a)$ is expressed by the Koszul rule $x\otimes y\mapsto (-1)^{ab} y\otimes x$.
\begin{lemma}\label{lemma:2.1}
Let $S^{\prime}\rightarrow S$ be a morphism of noetherian schemes and $(\pi^{\prime}:X^{\prime}\rightarrow S^{\prime};\sigma_{1}^{\prime},\ldots,\sigma_{n}^{\prime})$ the pointed stable curve deduced from $(\pi:X\rightarrow S;$ $\sigma_{1},\ldots,\sigma_{n})$ by base change to $S^{\prime}$. For chosen integers $1\leq j_{1}<\ldots<j_{l}\leq n$, put $\sigma=\sigma_{j_1}+\ldots+\sigma_{j_l}$ (resp. $\sigma^{\prime}=\sigma_{j_1}^{\prime}+\ldots+\sigma_{j_l}^{\prime}$). Then we have canonical isomorphisms of $\OO_{X^{\prime}}$-modules
\begin{displaymath}
	p^{\ast}(\OO_{\sigma})\overset{\sim}{\longrightarrow}\OO_{\sigma^{\prime}}
\end{displaymath}
and
\begin{displaymath}
	p^{\ast}(\OO_{X}(-\sigma)/\OO_{X}(-\sigma)^{2})\overset{\sim}{\longrightarrow}\OO_{X^{\prime}}(-\sigma^{\prime})/
	\OO_{X^{\prime}}(-\sigma^{\prime})^{2},
\end{displaymath}
where $p:X^{\prime}=X\times_{S}S^{\prime}\rightarrow X$ is the first projection.
\end{lemma}
\begin{proof}
Immediate from the flatness of the $\OO_{S}$-modules $\OO_{\sigma}=\OO_{X}/\OO_{X}(-\sigma)$ and $\OO_{X}(-\sigma)/\OO_{X}(-\sigma)^{2}$.
\end{proof}
\begin{proposition}\label{proposition:2.1}
Let $\sigma=\sigma_{j_1}+\ldots+\sigma_{j_{l}}$, for chosen indexes $1\leq j_{1}<\ldots<j_{l}\leq n$. Then there is a canonical isomorphism
\begin{displaymath}
	\Res_{\sigma}:\omega_{X/S}(\sigma)\otimes_{\OO_{X}}\OO_{\sigma}\overset{\sim}{\longrightarrow}\OO_{\sigma},
\end{displaymath}
compatible with base change.
\end{proposition}
\begin{proof}
First of all $\sigma_{1},\ldots,\sigma_{n}$ are disjoint sections of $\pi$ lying in the smooth locus. Therefore the differential $d_{X/S}:\OO_{X}\rightarrow\Omega_{X/S}$ \cite[Sec. 16.3]{EGA4} induces an isomorphism
\begin{displaymath}
	\widetilde{d}_{X/S}:\OO_{X}(-\sigma)/\OO_{X}(-\sigma)^{2}\overset{\sim}{\longrightarrow}\omega_{X/S}\otimes_{\OO_{X}}\OO_{\sigma}
\end{displaymath}
(see \cite[Prop. 17.2.5]{EGA4}). If $S^{\prime}\rightarrow S$ is a morphism of noetherian schemes, then there is a commutative diagram \cite[Sec. 16.4]{EGA4}
\begin{equation}\label{equation:2.2}
	\xymatrix{
	(\OO_{X}(-\sigma)/\OO_{X}(-\sigma)^{2})\otimes_{\OO_{S}}\OO_{S^{\prime}}\ar[r]^{\hspace{0.5cm} \widetilde{d}_{X/S}}\ar[d]_{\wr} & (\omega_{X/S}\otimes_{\OO_{X}}\OO_{\sigma})\otimes_
	{\OO_{S}}\OO_{S^{\prime}}\ar[d]^{\wr}\\
	\OO_{X^{\prime}}(-\sigma^{\prime})/\OO_{X^{\prime}}(-\sigma^{\prime})^{2}\ar[r]^{\hspace{0.5cm} \widetilde{d}_{X^{\prime}/S^{\prime}}} & \omega_{X^{\prime}/S^{\prime}}
	\otimes_{\OO_{X^{\prime}}}\OO_{\sigma^{\prime}}.
	}
\end{equation}
The vertical arrows are isomorphisms due to Lemma \ref{lemma:2.1} and the compatibility of the relative dualizing sheaf with base change. 

To conclude, we notice the canonical isomorphism
\begin{equation}\label{equation:2.3}
	\OO_{\sigma}\otimes_{\OO_{X}}\OO_{X}(-\sigma)\overset{\sim}{\longrightarrow}\OO_{X}(-\sigma)/\OO_{X}(-\sigma)^{2},	
\end{equation}
compatible with base change. The morphism $\Res_{\sigma}$ is then obtained from $\widetilde{d}_{X/S}^{-1}$, tensoring by $\OO_{X}(\sigma)$ and taking (\ref{equation:2.3}) into account. The compatibility of $\Res_{\sigma}$ with base change follows from (\ref{equation:2.2}).
\end{proof}
\begin{proposition}\label{proposition:2.2}
Let $k\geq 1$ be an integer and $\sigma=\sigma_{1}+\ldots+\sigma_{n}$.\\
i. There are exact sequences of locally free sheaves on $S$
\begin{equation}\label{equation:2.4}
	\begin{split}
	&0\longrightarrow\pi_{\ast}(\omega_{X/S}^{k+1}(k\sigma_{1}+\ldots+k\sigma_{n}))\\
	&\hspace{1cm}\longrightarrow\pi_{\ast}(\omega_{X/S}^{k+1}(k\sigma_{1}+\ldots+(k+1)\sigma_{j}+\ldots+k\sigma_{n}))
	\overset{\pi_{\ast}\Res_{\sigma_j}^{k+1}}{\longrightarrow}\OO_{S}\longrightarrow 0
	\end{split}
\end{equation}
and
\begin{equation}\label{equation:2.5}
	\begin{split}
	&0\longrightarrow\pi_{\ast}(\omega_{X/S}^{k+1}(k\sigma_{1}+\ldots+k\sigma_{n}))\\
	&\hspace{1cm}\longrightarrow\pi_{\ast}(\omega_{X/S}(\sigma_{1}+\ldots+\sigma_{n})^{k+1})
	\overset{\pi_{\ast}\Res_{\sigma}^{k+1}}{\longrightarrow}\OO_{S}^{\oplus n}\longrightarrow 0,	
	\end{split}
\end{equation}
compatible with base change.\\
ii. Let $\delta_{j}:\OO_{S}\rightarrow\OO_{S}^{\oplus n}$ be the inclusion into the $j$-th factor. There is a natural commutative diagram
\begin{displaymath}
	\xymatrix{
		\pi_{\ast}(\omega_{X/S}^{k+1}(k\sigma_{1}+\ldots+k\sigma_{n}))\ar[r]^{\hspace{2.2cm}\pi_{\ast}\Res_{\sigma_{j}}^{k+1}}\ar[d]	&\OO_{S}\ar[d]^{\delta_{j}}\\
		\pi_{\ast}(\omega_{X/S}(\sigma_{1}+\ldots+\sigma_{n})^{k+1})\ar[r]^{\hspace{2.2cm}\pi_{\ast}\Res_{\sigma}^{k+1}}	&\OO_{S}^{\oplus n},
	}
\end{displaymath}
compatible with base change.
\end{proposition}
\begin{proof}
\textit{i}. Let us establish the exact sequence (\ref{equation:2.5}). We leave the proof of (\ref{equation:2.4}) as an analogous exercise. By Proposition \ref{proposition:2.1}, we have an exact sequence of coherent sheaves on $X$
\begin{equation}\label{equation:2.6}
\begin{split}
	&0\longrightarrow\omega_{X/S}^{k+1}(k\sigma_{1}+\ldots+k\sigma_{n})\\
	&\hspace{1cm}\longrightarrow\omega_{X/S}(\sigma_{1}+\ldots+\sigma_{n})^{k+1}
	\overset{\Res_{\sigma}^{k+1}}{\rightarrow}\OO_{\sigma}\longrightarrow 0,
	\end{split}
\end{equation}
commuting with base change. The first two sheaves in (\ref{equation:2.6}) are flat over $S$ and have vanishing higher direct images $R^{1}\pi_{\ast}$. Then, from \cite[Cor. 1.5]{Knudsen}, $\pi_{\ast}(\omega_{X/S}^{k+1}(k\sigma_{1}+\ldots+k\sigma_{n}))$ and $\pi_{\ast}(\omega_{X/S}(\sigma_{1}+\ldots+\sigma_{n})^{k+1})$ are locally free and commute with base change. Furthermore $\pi_{\ast}\OO_{\sigma}=\OO_{S}^{\oplus n}$. Therefore, for any morphism $p:S^{\prime}\rightarrow S$ of noetherian schemes, we have a commutative diagram of exact sequences
\begin{displaymath}
	\xymatrix{
		0\ar[r]	&p^{\ast}\pi_{\ast}(\omega_{X/S}^{k+1}(k\sigma_{1}+\ldots+k\sigma_{n}))\ar[d]_{\wr}^{\alpha}\ar[r]
			&p^{\ast}\pi_{\ast}(\omega_{X/S}(\sigma_{1}+\ldots+\sigma_{n})^{k+1})\ar[d]_{\wr}^{\beta}\\
		0\ar[r]	&\pi^{\prime}_{\ast}(\omega_{X^{\prime}/S^{\prime}}^{k+1}(k\sigma_{1}+\ldots+k\sigma_{n}))\ar[r]
			&\pi^{\prime}_{\ast}(\omega_{X^{\prime}/S^{\prime}}(\sigma_{1}^{\prime}+\ldots+\sigma_{n}^{\prime})^{k+1})\\
	}
\end{displaymath}
\begin{displaymath}
	\xymatrix{
		&\hspace{5.5cm}	&\ar[r]^{\hspace{-1cm}p^{\ast}\pi_{\ast}\Res_{\sigma}}	&\OO_{S^{\prime}}^{\oplus n}\ar[r]\ar@{-->}[d]^{\gamma} &0\\
		&\hspace{5.5cm}	&\ar[r]^{\hspace{-1cm}\pi^{\prime}_{\ast}\Res_{\sigma^{\prime}}}	&\OO_{S^{\prime}}^{\oplus n}\ar[r]	&0.
	}
\end{displaymath}
Observe that the vertical arrow $\gamma$ is completely determined by $\alpha$ and $\beta$. In addition, by the five lemma, $\gamma$ is an isomorphism. This completes the proof of the first item.\\
\textit{ii}. The second item is clear from the definition of $\Res_{\sigma_{j}}$ and $\Res_{\sigma}$.
\end{proof}
\begin{corollary}\label{corollary:2.1}
Let $k\geq 0$ be an integer. There are canonical isomorphisms of graded invertible sheaves on $S$
\begin{equation}\label{equation:2.7}
	\begin{split}
	&\det R\pi_{\ast}(\omega_{X/S}^{k+1}(k\sigma_{1}+(k+1)\sigma_{j}+\ldots+k\sigma_{n}))\\
	&\hspace{1cm}\overset{\sim}{\longrightarrow}\det R\pi_{\ast}(\omega_{X/S}^{k+1}(k\sigma_{1}+\ldots+k\sigma_{n}))
	\otimes_{\OO_{S}} (\OO_{S},1)
	\end{split}
\end{equation}
and
\begin{equation}\label{equation:2.8}
	\begin{split}
	&\det R\pi_{\ast}(\omega_{X/S}(\sigma_{1}+\ldots+\sigma_{n})^{k+1})\\
	&\hspace{1cm}\overset{\sim}{\longrightarrow}\det R\pi_{\ast}(\omega_{X/S}^{k+1}(k\sigma_{1}+\ldots+k\sigma_{n}))
	\otimes_{\OO_{S}} (\OO_{S},n),
	\end{split}
\end{equation}
compatible with base change.
\end{corollary}
\begin{proof}
For $k\geq 1$, the corollary already follows from Proposition \ref{proposition:2.2} and the theory of determinants of Knudsen-Mumford \cite[Ch. I]{Knudsen-Mumford}. The case $k=0$ is individually handled. For the proof of (\ref{equation:2.8}), first of all Proposition \ref{proposition:2.1} provides an exact sequence
\begin{equation}\label{equation:2.9}
	0\longrightarrow\omega_{X/S}\longrightarrow\omega_{X/S}(\sigma_{1}+\ldots+\sigma_{n})
	\overset{\Res_{\sigma}}{\longrightarrow}\OO_{\sigma}\longrightarrow 0.
\end{equation}
By \cite[Ch. I, p. 46]{Knudsen-Mumford}, attached to (\ref{equation:2.9}) there is a canonical isomorphism of graded invertible sheaves
\begin{displaymath}
	\det R\pi_{\ast}\omega_{X/S}(\sigma_{1}+\ldots+\sigma_{n})\overset{\sim}{\longrightarrow}
	(\det R\pi_{\ast}\omega_{X/S})\otimes_{\OO_{S}}(\det R\pi_{\ast}\OO_{\sigma}),
\end{displaymath}
compatible with base change. Moreover $\det R\pi_{\ast}\OO_{\sigma}=(\OO_{S},n)$, canonically and commuting with base change. The claim follows. The argument for (\ref{equation:2.7}) is analogous.
\end{proof}
\begin{remark}
\textit{i}. The preceding considerations can be effected in the category of complex analytic spaces. The previous results are then compatible with the analytification functor from the category of schemes of finite type over $\CC$ to the category of complex analytic spaces.\\
\textit{ii}. Let $(\pi:X\rightarrow\Spec\CC;\sigma_{1},\ldots,\sigma_{n})$ be a pointed stable curve. By Chow's lemma we may identify $X$ with its associated complex analytic curve $X^{\an}$. For every $j=1,\ldots, n$, let $u_{j}$ be an analytic coordinate in a neighborhood of $\sigma_{j}$, with $u_{j}(\sigma_{j})=0$. Consider $\theta\in H^{0}(X,\omega_{X}(\sigma_{1}+\ldots+\sigma_{n})^{k+1})$ and express $\theta=\theta_{j}(u_{j})(du_{j}/u_{j})^{k+1}$ in the coordinate $u_{j}$, with $\theta_{j}$ holomorphic at $0$. Then, for $\sigma=\sigma_{1}+\ldots+\sigma_{n}$,
\begin{displaymath}
	\pi_{\ast}\Res_{\sigma}^{k+1}\theta=(\theta_{1}(0),\ldots,\theta_{n}(0)).
\end{displaymath}
Observe that --as expected-- this is coordinate independent. The analogue expression for $\pi_{\ast}\Res_{\sigma_{j}}^{k+1}$ is also true.
\end{remark}
\subsection{Clutching morphisms and determinants}
Let $(\pi_{1}:X_{1}\rightarrow S;\sigma_{1},\ldots,\sigma_{n+1})$ and $(\pi_{2}:X_{2}\rightarrow S;\tau_{1},\ldots,\tau_{m+1})$ be two pointed stable curves over a noetherian scheme $S$. Knudsen's clutching construction \cite[Th. 3.4 and Def. 3.6]{Knudsen} applied to the prestable curve $X_{1}\sqcup X_{2}\rightarrow S$, together with the sections $\sigma_{n+1}$, $\tau_{m+1}$, produces a new prestable curve $\pi:X\rightarrow S$. The sections $\sigma_{1},\ldots,\sigma_{n}$, $\tau_{1},\ldots,\tau_{m}$ define sections of $\pi$. The tuple $(\pi:X\rightarrow S;\sigma_{1},\ldots,\sigma_{n},\tau_{1},\ldots,\tau_{m})$ is a pointed stable curve.
\begin{proposition}\label{proposition:2.3}
Let $k\geq 1$ be an integer. Suppose that $S$ is integral. Then there exists a canonical isomorphism of graded invertible sheaves
\begin{displaymath}
	\begin{split}
	&\det R\pi_{\ast}\omega_{X/S}^{k+1}(k\sigma_{1}+\ldots+k\sigma_{n}+k\tau_{1}+\ldots+k\tau_{m})\\
	&\hspace{1cm}\overset{\sim}{\longrightarrow}(\det R\pi_{1\ast}\omega_{X_{1}/S}^{k+1}(k\sigma_{1}+\ldots+k\sigma_{n+1}))\\
	&\hspace{2cm}\otimes_{\OO_{S}}(\det R\pi_{2\ast}\omega_{X_{2}/S}^{k+1}(k\tau_{1}+\ldots+k\tau_{m+1}))\otimes_{\OO_{S}}
	(\OO_{S},1),
	\end{split}
\end{displaymath}
compatible with base change.
\end{proposition}
The proof of the proposition will reduce to the construction described below.
\begin{construction}\label{construction:1}
\textit{i}. By \cite[Th. 3.4]{Knudsen}, $X$ comes equipped with finite morphisms of $S$-schemes
\begin{displaymath}
	\iota_{1}:X_{1}\longrightarrow X,\quad\iota_{2}:X_{2}\longrightarrow X.
\end{displaymath}
By the properties of the relative dualizing sheaf \cite[Sec. 1]{Knudsen}, there is a natural morphism of coherent sheaves
\begin{equation}\label{equation:2.10}
	\begin{split}
	&\omega_{X/S}^{k+1}(k\sigma_{1}+\ldots+k\sigma_{n}+k\tau_{1}+\ldots+k\tau_{m})\\
	&\hspace{1cm}\longrightarrow\iota_{1\ast}\omega_{X_{1}/S}^{k+1}(k\sigma_{1}+\ldots+k\sigma_{n}+(k+1)\sigma_{n+1})\\
	&\hspace{2cm}\oplus\iota_{2\ast}\omega_{X_{2}/S}^{k+1}(k\tau_{1}+\ldots+k\tau_{n}+(k+1)\tau_{m+1}).
	\end{split}
\end{equation}
If $k\geq 1$, applying $\pi_{\ast}$ to (\ref{equation:2.10}), we deduce a morphism of locally free coherent sheaves
\begin{displaymath}
	\begin{split}
	&\varphi:\pi_{\ast}\omega_{X/S}^{k+1}(k\sigma_{1}+\ldots+k\sigma_{n}+k\tau_{1}+\ldots+k\tau_{m})\\
	&\hspace{1cm}\longrightarrow\pi_{1\ast}\omega_{X_{1}/S}^{k+1}(k\sigma_{1}+\ldots+k\sigma_{n}+(k+1)\sigma_{n+1})\\
	&\hspace{2cm}\oplus\pi_{2\ast}\omega_{X_{2}/S}^{k+1}(k\tau_{1}+\ldots+k\tau_{n}+(k+1)\tau_{m+1}).
	\end{split}
\end{displaymath}
Notice that the formation of $\varphi$ commutes with base change.\\
\textit{ii}. Replacing $S$ by a smaller Zariski open subset, we may assume that the exact sequences (\ref{equation:2.3}) for both $(\pi_{1};\sigma_{1},\ldots,\sigma_{n+1})$ (with $j=n+1$) and $(\pi_{2};\tau_{1},\ldots,\tau_{m+1})$ (with $j=m+1$) split. Let $\gamma_{1}$ be a lift of $1\in\OO_{S}$ by $\pi_{1\ast}\Res_{\sigma_{n+1}}^{k+1}$, and similarly introduce $\gamma_{2}$. We then have \textit{non-canonical} isomorphisms
\begin{equation}\label{equation:2.11}
	\begin{split}
		\psi_{1}:&\pi_{1\ast}\omega_{X_{1}/S}^{k+1}(k\sigma_{1}+\ldots+k\sigma_{n}+(k+1)\sigma_{n+1})\\
		&\hspace{1cm}\longrightarrow\pi_{1\ast}\omega_{X_{1}/S}^{k+1}(k\sigma_{1}+\ldots+k\sigma_{n+1})\oplus\OO_{S}\\
		&\hspace{0.8cm}\theta\longmapsto(\theta-(\pi_{1\ast}\Res_{\sigma_{n+1}}^{k+1}\theta)\gamma_{1}, \pi_{1\ast}\Res_{\sigma_{n+1}}^{k+1}\theta)
	\end{split}
\end{equation}
and an analogously constructed $\psi_{2}$.\\
\textit{iii}. Define $\Phi=(\psi_{1}\oplus\psi_{2})\circ\varphi$ and $\widetilde{\Phi}$ the composite of $\Phi$ with the projection onto the factor $\OO_{S}\oplus\OO_{S}$. We claim that $\Imag\widetilde{\Phi}$ is contained in the $\OO_{S}$-submodule $\Delta:=\lbrace (x,(-1)^{k+1}x)\in\OO_{S}\oplus\OO_{S}\rbrace$. This is a local assertion: we first reduce to $S$ being an affine local noetherian scheme. Secondly, by Nakayama's lemma, we may assume that $S=\Spec k$, for $k$ an algebraically closed field. In this case the claim follows by the properties of the dualizing sheaf \cite[Sec. 1]{Knudsen} and the very construction of $\Phi$. This completes the proof. Consequently, in the definition of $\Phi$, we are allowed to replace $\OO_{S}\oplus\OO_{S}$ by the $\OO_{S}$-module $\Delta\simeq\OO_{S}$.
\end{construction}
\begin{lemma}\label{lemma:2.2}
Suppose that $S$ is integral. Then $\Phi$ is an isomorphism.
\end{lemma}
\begin{proof}
To check that $\Phi$ is an isomorphism, we may assume that $S$ is an affine local scheme.\\
\textit{Injectivity}. Since $S$ is integral, we reduce to the case $S=\Spec k$, where $k$ is an algebraically closed field. The injectivity is then clear by construction of $\Phi$.\\
\textit{Surjectivity}. By Nakayama's lemma, we reduce again to the case $S=\Spec k$, where $k$ is an algebraically closed field. The surjectivity of $\Phi$ is then a consequence of the properties of the relative dualizing sheaf \cite[Sec. 1]{Knudsen}.
\end{proof}
We are ready to complete the proof of Proposition \ref{proposition:2.3}.
\begin{proof}[Proof of Proposition \ref{proposition:2.3}]
Let $\lbrace U_{i}\rbrace_{i}$ be a finite open covering of $S$, such that over every $U_{i}$ there exists an isomorphism $\Phi_{i}$ as in Construction \ref{construction:1}. The isomorphisms $\Phi_{i}$ depend on the choices $(\gamma_{1}^{(i)}$, $\gamma_{2}^{(i)})$. Since $S$ is integral, the $\Phi_{i}$ induce isomorphisms
\begin{displaymath}
	\begin{split}
	\det\Phi_{i}: &\det R\pi_{\ast}\omega_{X/S}^{k+1}(k\sigma_{1}+\ldots+k\sigma_{n}+k\tau_{1}+\ldots+k\tau_{m})\mid_{U_i}\\
	&\hspace{1cm}\overset{\sim}{\longrightarrow}(\det R\pi_{1\ast}\omega_{X_{1}/S}^{k+1}(k\sigma_{1}+\ldots+k\sigma_{n+1}))\\
	&\hspace{2cm}\otimes_{\OO_{S}}(\det R\pi_{2\ast}\omega_{X_{2}/S}^{k+1}(k\tau_{1}+\ldots+k\tau_{m+1}))
	\otimes_{\OO_{S}}(\OO_{S},1)\mid_{U_{i}}.
	\end{split}
\end{displaymath}
One easily checks that the isomorphisms $\det\Phi_{i}$ do \textit{not} depend on the choices $(\gamma_{1}^{(i)},\gamma_{2}^{(i)})$. By construction, the $\det\Phi_{i}$ commute with base change. The collection $\lbrace\det\Phi_{i}\rbrace_{i}$ can be glued into a single canonical isomorphism defined over $S$, compatible with base change. The proposition is proven. 
\end{proof}
\begin{remark}
The preceding constructions can be made in the category of analytic spaces. They are preserved by the analytification functor from the category of schemes of finite type over $\CC$ to the category of analytic spaces.
\end{remark}
\section{Tautological line bundles and Mumford isomorphisms}\label{section:tautological}
Let $g,n \geq 0$ be integers with $2g-2+n>0$ and $\SCM_{g,n}\rightarrow\Spec\Int$ the Deligne-Mumford algebraic stack classifying $n$-pointed stable curves of genus $g$ \cite{Knudsen}. We let $\pi:\SCC_{g,n}\rightarrow\SCM_{g,n}$ be the universal curve and $\sigma_{1},\ldots,\sigma_{n}$ the universal sections of $\pi$. We proceed to introduce the tautological line bundles on $\SCM_{g,n}$ and establish relations between them: the Mumford isomorphisms.
\subsection{Tautological line bundles on $\SCM_{g,n}$}
\begin{definition}
The \textit{tautological line bundles} on $\SCM_{g,n}$ are
\begin{displaymath}
	\begin{split}
		&\lambda_{k+1;g,n}=\det R\pi_{\ast}\omega_{\SCC_{g,n}/\SCM_{g,n}}^{k+1}(k\sigma_{1}+\ldots+k\sigma_{n}),\quad k\geq 0,\\
		&\delta_{g,n}=\OO(\pd\SM_{g,n}),\\
		&\psi_{g,n}^{(j)}=\sigma_{j}^{\ast}\omega_{\SCC_{g,n}/\SCM_{g,n}},\quad j=1,\ldots,n,\\
		&\psi_{g,n}=\otimes_{j}\psi_{g,n}^{(j)},\\
		&\kappa_{g,n}=\langle\omega_{\SCC_{g,n}/\SCM_{g,n}}(\sigma_{1}+\ldots+\sigma_{n}),\omega_{\SCC_{g,n}/\SCM_{g,n}}(\sigma_{1}+\ldots+\sigma_{n})\rangle,
	\end{split}
\end{displaymath}
where $\langle\cdot,\cdot\rangle$ denotes the Deligne pairing \cite[XVIII]{SGA4}, \cite{Elkik}.
\end{definition}
\begin{proposition}\label{proposition:3.1}
Let $k\geq 0$ be an integer. Then there exists an isomorphism of line bundles on $\SCM_{g,n}$, uniquely determined up to a sign,
\begin{displaymath}
	\lambda_{k+1;g,n}\overset{\sim}{\longrightarrow}\det R\pi_{\ast}(\omega_{\SCC_{g,n}/\SCM_{g,n}}(\sigma_{1}+\ldots+\sigma_{n})^{k+1}).
\end{displaymath}
\end{proposition}

\begin{proof}
The existence of the isomorphism is a reformulation of Corollary \ref{corollary:2.1} (\ref{equation:2.8}). The uniqueness assertion follows from \cite[Cor. 3.2]{SingARR}.
\end{proof}
\begin{proposition}\label{proposition:3.2}
Let $\beta:\SCM_{g_{1},n_{1}+1}\times\SCM_{g_{2},n_{2}+1}\rightarrow\SCM_{g_{1}+g_{2},n_{1}+n_{2}}$ be Knudsen's clutching morphism \cite[Th. 3.4 and Def. 3.6]{Knudsen}. Then there are isomorphisms of line bundles on $\SCM_{g_{1},n_{1}+1}\times\SCM_{g_{2},n_{2}+1}$, uniquely determined up to a sign,
\begin{align}
		&\beta^{\ast}\lambda_{k+1;g,n}\overset{\sim}{\rightarrow}\lambda_{k+1;g_{1},n_{1}+1}\boxtimes\lambda_{k+1;g_{2},n_{2}+1},\quad k\geq 0,
		\label{equation:3.1}\\
		&\beta^{\ast}\delta_{g,n}\overset{\sim}{\rightarrow}(\delta_{g_{1},n_{1}+1}\otimes\psi^{(n_{1}+1)\,-1}_{g_{1},n_{1}+1})
		\boxtimes(\delta_{g_{2},n_{2}+1}\otimes\psi^{(n_{2}+1)\,-1}_{g_{2},n_{2}+1}),\label{equation:3.2}\\
		&\beta^{\ast}\psi_{g,n}\overset{\sim}{\rightarrow}(\psi_{g_{1},n_{1}+1}\otimes\psi_{g_{1},n_{1}+1}^{(n_{1}+1)\,-1})
		\boxtimes(\psi_{g_{2},n_{2}+1}\otimes\psi_{g_{2},n_{2}+1}^{(n_{2}+1)\,-1}),\label{equation:3.3}\\
		&\beta^{\ast}\kappa_{g,n}\overset{\sim}{\rightarrow}\kappa_{g_{1},n_{1}+1}\boxtimes\kappa_{g_{2},n_{2}+1}.\label{equation:3.4}
\end{align}
\end{proposition}
\begin{proof}
For the isomorphisms (\ref{equation:3.2})--(\ref{equation:3.4}) and (\ref{equation:3.1}) for $k=0$, we refer to \cite[Prop. 3.7]{SingARR}. The isomorphism (\ref{equation:3.1}) for $k\geq 1$ is deduced from Proposition \ref{proposition:2.3}. The uniqueness assertion comes from \cite[Cor. 3.2]{SingARR}.
\end{proof}
\begin{corollary}\label{corollary:3.1}
Let $\gamma:\SCM_{g,n}\times\SCM_{1,1}^{\times n}\rightarrow\SCM_{g+n,0}$ be obtained by reiterated applications of clutching morphisms. Then we have isomorphisms, uniquely determined up to a sign,
\begin{align}
	&\gamma^{\ast}\lambda_{k+1;g+n,0}\overset{\sim}{\rightarrow}\lambda_{k+1;g,n}\boxtimes\lambda_{k+1;1,1}
	^{\boxtimes n},\quad k\geq 0,\label{equation:3.5}\\
	&\gamma^{\ast}\delta_{g+n,0}\overset{\sim}{\rightarrow}(\delta_{g,n}\otimes\psi_{g,n}^{-1})\boxtimes
	(\delta_{1,1}\otimes\psi_{1,1}^{-1})^{\boxtimes n},\label{equation:3.6}\\
	&\gamma^{\ast}\kappa_{g+n,0}
	\overset{\sim}{\rightarrow}\kappa_{g,n}\boxtimes\kappa_{1,1}^{\boxtimes n}.\label{equation:3.7}
\end{align}
\end{corollary}
\begin{proof}
This is straightforward from Proposition \ref{proposition:3.2} and \cite[Cor. 3.2]{SingARR}.
\end{proof}
\subsection{Mumford isomorphisms}
\begin{theorem}\label{theorem:3.1}
Let $k\geq 0$ be an integer. There is an isomorphism of invertible sheaves on $\SCM_{g,n}$, uniquely determined up to a sign,
\begin{displaymath}
	\DD_{k+1;g,n}:\lambda_{k+1;g,n}^{\otimes 12}\otimes\delta_{g,n}^{-1}\otimes\psi_{g,n}\overset{\sim}{\longrightarrow}\kappa_{g,n}^{\otimes (6k^{2}+6k+1)}.
\end{displaymath}
\end{theorem}
The proof of the theorem relies on \cite[Th. 3.10]{SingARR} and the following lemma.
\begin{lemma}\label{lemma:3.1}
Let $k\geq 1$ be an integer. There is an isomorphism of invertible sheaves on $\SCM_{g,n}$
\begin{equation}\label{equation:3.8}
	\lambda_{k+1;g,n}\overset{\sim}{\longrightarrow}\lambda_{k;g,n}\otimes\kappa_{g,n}^{\otimes k}.
\end{equation}
\end{lemma}
\begin{proof}
We proceed in two steps.\\
\textit{Step 1}. We first treat the cases $g=n=1$ or $g\geq 2$, $n=0$. Let $\SM=\SCM_{1,1}$ or $\SCM_{g,0}$, $\C=\SCC_{1,1}$ or $\SCC_{g,0}$. Recall the cohomological expression for the Deligne-Weil pairing \cite[Exp. II]{Szpiro}--\cite{MB}: given line bundles $L,M$ on $\C$, there is an isomorphism of invertible sheaves
\begin{equation}\label{equation:3.9}
	\begin{split}
	\langle L,M\rangle\simeq&\det R\pi_{\ast}(L\otimes M)\\
	&\otimes(\det R\pi_{\ast}L)^{-1}\otimes(\det R\pi_{\ast}M)^{-1}\otimes\det R\pi_{\ast}\OO_{\C}.
	\end{split}
\end{equation}
If $g\geq 2$ and $n=0$, then (\ref{equation:3.9}) applied to $L=\omega_{\C/\SM}^{\otimes k}$, $M=\omega_{\C/\SM}$ and Serre's duality already prove (\ref{equation:3.8}). If $g=n=1$, we set $L=\omega_{\C/\SM}(\sigma)^{k}$ and $M=\omega_{\C/\SM}(\sigma)$. The conclusion follows by combination of (\ref{equation:3.9}), Proposition \ref{proposition:3.1} and Serre's duality.\\
\textit{Step 2}. For the general case, introduce the clutching morphism $\gamma:\SCM_{g,n}\times\SCM_{1,1}^{\times n}\rightarrow\SCM_{g+n,0}$. From Corollary \ref{corollary:3.1} and the first step, there are isomorphisms
\begin{equation}\label{equation:3.10}
	\lambda_{k+1;g,n}\boxtimes\lambda_{k+1;1,1}^{\boxtimes n}\overset{\sim}{\longrightarrow}(\lambda_{k;g,n}\otimes\kappa_{g,n}^{\otimes k})\boxtimes(\lambda_{k;1,1}\otimes\kappa_{1,1}^{\otimes k})^{\boxtimes n}
\end{equation}
and
\begin{equation}\label{equation:3.11}
	pr_{2}^{\ast}\lambda_{k+1;1,1}^{\boxtimes n}\overset{\sim}{\longrightarrow}
	pr_{2}^{\ast}(\lambda_{k;1,1}\otimes\kappa_{1,1}^{\otimes k})^{\boxtimes n},
\end{equation}
where $pr_{2}:\SCM_{g,n}\times\SCM_{1,1}^{\times n}\rightarrow\SCM_{1,1}^{\times n}$ is the projection onto the second factor. The isomorphisms (\ref{equation:3.10})--(\ref{equation:3.11}) together provide an isomorphism
\begin{displaymath}
	pr_{1}^{\ast}\lambda_{k+1;g,n}\overset{\sim}{\longrightarrow}pr_{1}^{\ast}(\lambda_{k;g,n}\otimes
	\kappa_{g,n}^{\otimes k}).
\end{displaymath}
We come up to the conclusion as in the proof of \cite[Th. 3.10]{SingARR}.
\end{proof}
\begin{proof}[Proof of Theorem \ref{theorem:3.1}]
The proof is by induction. The case $k=0$ is the content of \cite[Th. 3.10]{SingARR}. The induction step is achieved by means of Lemma \ref{lemma:3.1}. Uniqueness up to a sign is ensured by \cite[Cor. 3.2]{SingARR}.
\end{proof}
\begin{corollary}\label{corollary:3.2}
Let $\gamma:\SCM_{g,n}\times\SCM_{1,1}^{\times n}\rightarrow\SCM_{g+n,0}$ be a reiterated clutching morphism. Then the isomorphisms $\gamma^{\ast}\DD_{k+1;g+n,0}$ and $\DD_{k+1;g,n}\boxtimes\DD_{k+1;1,1}^{\boxtimes n}$ are equal up to a sign.
\end{corollary}
\begin{proof}
Combine Theorem \ref{theorem:3.1} and Corollary \ref{corollary:3.1}.
\end{proof}
\section{Degeneracy of the Quillen metric}\label{section:quillen}
Let $g,n\geq 0$ be integers with $2g-2+n>0$. Let $\gamma:\SCM_{g,n}\times\SCM_{1,1}\rightarrow\SCM_{g+n,0}$ be a reiterated clutching morphism. The object of study of this section is the behavior of the Quillen metric on $\lambda_{k+1;g+n,0}^{\an}$ near the image of $\gamma^{\an}$ in $\SCM_{g+n,0}^{\an}$. Before the statement of the main theorem we introduce some notations that will prevail until the end.

Let $(X;a_{1},\ldots,a_{n})$ be a smooth $n$-pointed stable complex curve of genus $g$, and $(T_{1};b_{1})$, $\ldots$, $(T_{n};b_{n})$ $n$ smooth 1-pointed stable complex curves of genus $1$. These data define complex valued points $P\in\SCM_{g,n}(\CC)$, $Q_{1},\ldots,Q_{n}\in\SCM_{1,1}(\CC)$. The image of $N=(P_{1},Q_{1},\ldots,Q_{n})$ by the reiterated clutching morphism $\gamma:\SCM_{g,n}\times\SCM_{1,1}^{\times n}\rightarrow\SCM_{g+n,0}$ is a complex valued point $R\in\SCM_{g+n,0}(\CC)$. Notice that $R$ represents the quotient analytic space
\begin{displaymath}
	Y=(X\sqcup T_{1}\sqcup\ldots\sqcup T_{n})/(a_{1}\sim b_{1},\ldots,a_{n}\sim b_{n}).
\end{displaymath}
The compactness of $Y$ ensures its algebraicity. Following \cite[Cons. 4.1]{SingARR}, we construct a small stable deformation of $Y$, to be denoted $\ff:\fY\rightarrow\Omega$. With the notations of loc. cit., we build the family $\fg:\fZ\rightarrow D$ by restriction of $\ff$ to the locus $s_{1}=\ldots=s_{r}=0$ and $t_{1}=\ldots=t_{n}=t\in D\subset\CC$. The fibers $\fg^{-1}(t)=\fZ_{t}$, with $t\neq 0$, are smooth complex compact curves, whereas $Y=\fZ_{0}$ is singular.

Let $k\geq 1$ be an integer. We construct and deform a base of $H^{0}(Y,\omega_{Y}^{k+1})$, according to Proposition \ref{proposition:2.1} and Construction \ref{construction:1}.
\begin{construction}\label{construction:2}
\textit{i}. Fix bases $\alpha_{1},\ldots,\alpha_{a}$ of $H^{0}(X,\omega_{X}^{k+1}(ka_{1}+\ldots+ka_{n}))$ and $\beta_{1}^{(j)},\ldots,\beta_{b}^{(j)}$ of $H^{0}(Y_{j},\omega_{T_{j}}^{k+1}(kb_{j}))$. Notice that $\Res_{a_{i}}^{k+1}\alpha_{j}=0$ for all $i,j$ and $\Res_{b_{j}}^{k+1}\beta_{l}^{(j)}=0$ for all $j,l$. Therefore, extending the forms $\alpha_{i},\beta_{l}^{(j)}$ by $0$ produces global sections of $\omega_{Y}^{k+1}$. We use the same symbols to refer to these extensions.\\
\textit{ii}. For every $j=1,\ldots,n$, choose a form $\gamma_{j}\in H^{0}(X,\omega_{X}^{k+1}(ka_{1}+\ldots+(k+1)a_{j}+\ldots+ka_{n})$ with $\Res_{a_{j}}^{k+1}\gamma_{j}=1$. Also, let $\phi_{j}\in H^{0}(T_{j},\omega_{T_{j}}(b_{j})^{k+1})$ be a form with $\Res_{b_{j}}^{k+1}\phi_{j}=(-1)^{k+1}$. By Construction \ref{construction:1} we can glue the forms $\gamma_{j},\phi_{j}$ into a global section $\theta_{j}$ of $\omega_{Y}^{k+1}$.\\
\textit{iii}. By construction, $\alpha_{i}$, $\beta_{l}^{j}$, $\theta_{k}$, for all $i,j,k,l$ form a base of $H^{0}(Y,\omega_{Y}^{k+1})$.\\
\textit{iv}. The sheaf $R\fg_{\ast}\omega_{\fZ/D}^{k+1}$ is locally free. Hence, after possibly shrinking $\Omega$, we can extend the family of forms $\alpha_{i},\beta_{l}^{(j)},\theta_{k}$ to a frame $\alpha_{i}(t),\beta_{l}^{(j)}(t),\theta_{k}(t)$ of $R\fg_{\ast}\omega_{\fZ/D}$.
\end{construction}
\begin{lemma}\label{lemma:4.1}
Let $\Psi:\gamma^{\ast}\lambda_{k+1;g+n,0}\overset{\sim}{\rightarrow}\lambda_{k+1;g,n}\boxtimes\lambda_{k+1;1,1}^{\boxtimes n}$ be the isomorphism (\ref{equation:3.5}). Then $\Psi$ induces an isomorphism
\begin{displaymath}
	\begin{split}
	&\det H^{0}(\fZ_{0},\omega_{\fZ_{0}}^{k+1})\overset{\sim}{\longrightarrow}\det H^{0}(X,\omega_{X}^{k+1}(ka_{1}+\ldots+ka_{n}))\\
	&\hspace{7cm}\otimes\bigotimes_{j=1}^{n}\det H^{0}(T_{j},\omega_{T_{j}}^{k+1}(kb_{j}))\\
		&\theta_{1}\wedge\ldots\wedge\theta_{n}\wedge\alpha_{1}\wedge\ldots\wedge\beta_{b}^{(n)}\longmapsto
		\pm \alpha_{1}\wedge\ldots\wedge\alpha_{n}\\
		&\hspace{7cm}\otimes\beta_{1}^{(1)}\wedge\ldots\otimes\ldots\otimes\ldots\wedge\beta_{b}^{(n)}.
	\end{split}
\end{displaymath}
\end{lemma}
\begin{proof}
This is easily checked from Construction \ref{construction:1} and the proof of Proposition \ref{proposition:2.3}.
\end{proof}
Attached to the family $\fg:\fZ\rightarrow D$, there is a classifying map
\begin{displaymath}
	\C(\fg):D\rightarrow\SCM_{g+n,0}^{\an}.
\end{displaymath}
The line bundle $\det R\fg_{\ast}\omega_{\fZ/D}^{k+1}=\C(\fg)^{\ast}\lambda_{k+1;g+n,0}^{\an}$ is endowed with the family Quillen metric $\|\cdot\|_{Q}$. This metric is defined and smooth on $D\setminus\lbrace 0\rbrace$.
\begin{theorem}\label{theorem:4.1}
Let $k\geq 1$ be an integer and $\fg:\fZ\rightarrow D$ a one dimensional stable deformation of $Y$ as above. Let $\alpha_{i}(t),\beta_{l}^{(j)}(t),\theta_{k}(t)$ be a frame of $R\fg_{\ast}\omega_{\fZ/D}^{k+1}$ as in Construction \ref{construction:2}. Then we have
\begin{displaymath}
	\begin{split}
	&\lim_{\substack{t\to 0\\ t\neq 0}}\|\theta_{1}(t)\wedge\ldots\wedge\theta_{n}(t)\wedge\alpha_{1}(t)\wedge\ldots\wedge\beta_{b}^{(n)}(t)\|_{Q,t}|t|^{n/6}=\|\alpha_{1}\wedge\ldots\wedge\alpha_{a}\|_{Q}\\
	&\hspace{5cm}\cdot\|\beta_{1}^{(1)}\wedge\ldots\wedge\beta_{b}^{(1)}\|_{Q}\cdot\ldots\cdot\|\beta_{1}^{(n)}\wedge\ldots\wedge\beta_{b}^{(n)}\|_{Q}.
	\end{split}
\end{displaymath}
\end{theorem}
The proof of the theorem is detailed throughout the next subsections.
\subsection{Degeneracy of the Selberg zeta function in family}
\begin{theorem}[Wolpert]\label{theorem:4.2}
i. Let $\gamma_{1}(t),\ldots,\gamma_{n}(t)\subset\fZ_{t}$ be the simple closed geodesics that are pinched to a node as $t\to 0$. Then the holomorphic function
\begin{displaymath}
	\frac{Z(\fZ_{t},s)}{\prod_{j}Z_{l(\gamma_{j}(t))}(s)},\quad \Real s>\frac{1}{2}
\end{displaymath}
locally uniformly converges to $Z(X^{\circ},s)\prod_{j}Z(T_{j}^{\circ},s)$ as $t\to 0$.\\
ii. The length $l(\gamma_{j}(t))$ of the geodesic $\gamma_{j}(t)$ satisfies the estimate
\begin{displaymath}
	l(\gamma_{j}(t))=\frac{2\pi^{2}}{\log|t|^{-1}}+\OO\left(\frac{1}{(\log|t|)^{4}}\right),\quad\text{as}\quad t\to 0.
\end{displaymath}
\end{theorem}
\begin{proof}
The first item is contained in \cite[Th. 35 and Th. 38]{Schulze}, \cite[Proof of Conj. 1 and Conj. 2]{Wolpert:Selberg}. The second item is \cite[Ex. 4.3, p. 446]{Wolpert:hyperbolic}.
\end{proof}
\begin{corollary}\label{corollary:4.1}
Let $k\geq 1$ be an integer. Then the Selberg zeta functions of the fibers $\fZ_{t}$ of the family $\fg:\fZ\rightarrow D$ satisfy
\begin{displaymath}
	\begin{split}
	\lim_{\substack{t\to 0\\ t\neq 0}}Z(\fZ_{t},k+1)|t|^{-n/6}(\log|t|^{-1})^{-(2k+1)n}=&\frac{1}{(\pi (2\pi^{2})^{2k}(k!)^{2})^{n}}\\
	&\cdot Z(X^{\circ},k+1)\prod_{j=1}^{n}Z(T_{j}^{\circ},k+1).
	\end{split}
\end{displaymath}
\end{corollary}
\begin{proof}
First of all, from \cite[Lemma 39]{Schulze}, for $\Real s>0$ we have
\begin{displaymath}
	\lim_{l\to 0}\Gamma(s)^{2}Z_{l}(s)\exp(\pi^{2}/3l)l^{2s-1}=2\pi.
\end{displaymath}
We evaluate at $s=k+1$ and take Theorem \ref{theorem:4.2} \textit{ii} into account. We find the equivalent
\begin{equation}\label{equation:4.1}
	Z_{l(\gamma_{j}(t))}(k+1)\sim \frac{1}{\pi(2\pi^{2})^{2k}(k!)^{2}}|t|^{1/6}(\log|t|^{-1})^{2k+1},\quad\text{as}\quad t\to 0.
\end{equation}
The corollary is now a straightforward consequence of Theorem \ref{theorem:4.2} \textit{i} and (\ref{equation:4.1}).
\end{proof}
\subsection{Degeneracy of the family $L^{2}$ metric}
\begin{theorem}\label{theorem:4.3}
For the family $L^{2}$ metric on $\det R\pi_{\ast}\omega_{\fZ/D}^{k+1}\mid_{D\setminus\lbrace 0\rbrace}$ we have
\begin{displaymath}
	\begin{split}
	\lim_{\substack{t\to 0\\ t\neq 0}}\|\theta_{1}(t)\wedge\ldots\wedge&\theta_{n}(t)\wedge\alpha_{1}(t)\wedge\ldots\wedge\beta_{b}^{(n)}(t)\|_{L^{2},t}^{2}(\log|t|^{-1})^{-n(2k+1)}\\
	=&\left(\frac{(2k)!}{2^{k-1}\pi^{2k}(k!)^{2}}\right)^{n}\|\alpha_{1}\wedge\ldots\wedge\alpha_{a}\|_{L^{2}}^{2}\\
	&\cdot\|\beta_{1}^{(1)}\wedge\ldots\wedge\beta_{b}^{(1)}\|_{L^{2}}^{2}\cdot\ldots\cdot
	\|\beta_{1}^{(n)}\wedge\ldots\wedge\beta_{b}^{(n)}\|_{L^{2}}^{2}.
	\end{split}
\end{displaymath}
\end{theorem}
The proof of the theorem is an adaptation of \cite[Proof of Th. 6]{Obitsu-Wolpert}. As such we don't provide full details here. The reader is referred to loc. cit. to deepen in the techniques.

Let $0<c<1$ be a small real constant and $t\in\CC$ with $|t|<c$. Define $A_{t}=\lbrace u\in\CC\mid |t|/c<|u|<c\rbrace$. We denote by $\langle\cdot,\cdot\rangle_{t}$ the hermitian metric on the holomorphic tangent bundle $T_{A_{t}}$ characterized by
\begin{equation}\label{equation:4.2a}
	\begin{split}
	&\langle\frac{\pd}{\pd u},\frac{\pd}{\pd u}\rangle_{t}=\frac{1}{2}\left(\frac{|u|\log|t|}{\pi}\sin\left(\frac{\pi\log|u|}{\log|t|}\right)\right)^{-2}\quad\text{if}\quad t\neq 0,\\
	&\langle\frac{\pd}{\pd u},\frac{\pd}{\pd u}\rangle_{0}=\frac{1}{2}\frac{1}{|z|^{2}(\log|z|)^{2}}\quad\text{if}\quad t=0.
	\end{split}
\end{equation}
We still write $\langle\cdot,\cdot\rangle_{t}$ for the metric induced on the tensor powers $T_{A_{t}}^{\otimes k}$, $k\in\Int$. The normalized K\"ahler form attached to $\langle\cdot,\cdot\rangle_{t}$ is
\begin{equation}\label{equation:4.2b}
	\begin{split}
		\omega_{t}&=\frac{i}{2\pi}\langle\frac{\pd}{\pd u},\frac{\pd}{\pd u}\rangle_{t} du\wedge d\cu\\
		&=\frac{i}{4\pi}\left(\frac{|u|\log|t|}{\pi}\sin\left(\frac{\pi\log|u|}{\log|t|}\right)\right)^{-2}du\wedge d\cu\quad\text{if}\quad t\neq 0,\\
		\omega_{0}&=\frac{i}{4\pi}\frac{du\wedge d\cu}{|u|^{2}(\log|z|)^{2}}\quad\text{if}\quad t=0.
	\end{split}
\end{equation}
\begin{lemma}\label{lemma:4.2}
i. For $0<|t|<c$, there is an equality
\begin{displaymath}
	I_{1}(t):=\int_{A_{t}}\langle (du/u)^{k+1},(du/u)^{k+1}\rangle_{t}\omega_{t}=\frac{1}{2^{k-1}\pi^{2k}}(\log|u|^{-1})^{2k+1}\frac{(2k)!}{(k!)^{2}}R(t),
\end{displaymath}
where $R(t)\to 1$ as $t\to 0$.\\
ii. As $t\to 0$ we have
\begin{displaymath}
	I_{2}(t):=\int_{A_{t}}\langle u(du/u)^{k+1},(du/u)^{k+1}\rangle_{t}\omega_{t}=O(1).
\end{displaymath}
\end{lemma}
\begin{proof}
The changes $u=\rho e^{i\theta}$ and $x=(\log\rho)/\log|t|$ reduce the first integral to
\begin{displaymath}
	I_{1}(t)=\frac{2^{k+1}}{\pi^{2k}}(\log|t|^{-1})^{2k+1}\int_{\varepsilon}^{1-\varepsilon}(\sin(\pi x))^{2k}dx,
\end{displaymath}
where $\varepsilon=(\log c)/\log|t|$. We conclude thanks to the identitiy $\int_{0}^{1}(\sin(\pi x))^{2k}dx$ $=(2k)!/(2^{2k}(k!)^{2})$. The second integral is treated analogously.
\end{proof}
\begin{proposition}\label{proposition:4.1}
Let $0<|t|<c$. For the $L^{2}$-pairing $\langle\cdot,\cdot\rangle_{L^{2},t}$ on $H^{0}(\fZ_{t},\omega_{\fZ_{t}}^{k+1})$, we have as $t\to 0$:\\
i. $\langle\theta_{i}(t),\theta_{j}(t)\rangle_{L^{2},t}=\OO(1)$ if $i\neq j$;\\
ii. $\langle\theta_{i}(t),\alpha_{j}(t)\rangle_{L^{2},t}=\OO(1)$, for all $i,j$;\\
iii. $\langle\theta_{i}(t),\beta_{l}^{(j)}(t)\rangle_{L^{2},t}=\OO(1)$, for all $i,j,l$;\\
iv. $\langle\theta_{i}(t),\theta_{i}(t)\rangle_{L^{2},t}=\frac{1}{2^{k-1}\pi^{2k}}(\log|t|^{-1})^{2k+1}\frac{(2k)!}{(k!)^{2}}R(t)(1+\OO((\log|t|)^{-2}))$, for all $i$.
\end{proposition}
\begin{proof}
The proposition is established by combination of Lemma \ref{lemma:4.2}, \cite[Exp. 4.2]{Wolpert:hyperbolic} and the construction of the forms $\alpha_{i}(t),\beta_{l}^{(j)}(t),\theta_{k}(t)$ (Construction \ref{construction:2}). We refer to \cite[Proof of Th. 6, \textbf{The first term}]{Obitsu-Wolpert} for details.
\end{proof}
\begin{proposition}\label{proposition:4.2}
The functions $\langle\alpha_{i}(t),\alpha_{j}(t)\rangle_{L^{2},t}$, $\langle\alpha_{i}(t),\beta_{l}^{(j)}(t)\rangle_{L^{2},t}$, $\langle\beta_{l}^{(j)}(t),$ $\beta_{l'}^{(j')}(t)\rangle_{L^{2},t}$ continuously depend on $t\in D$.
\end{proposition}
\begin{proof}
Fix $\vartheta_{1}(t), \vartheta_{2}(t)\in\lbrace\alpha_{i}(t), \beta_{l}^{(j)}(t)\rbrace_{i,j,l}$. Corresponding to every node of $\fZ_{0}$, there is a holomorphic embedding of the collar $C_{t}:=\lbrace (u,v)\in\CC^{2}\mid |t|/c\leq |u|, |v|\leq c,\, uv=t\rbrace$ into $\fZ_{t}$, provided $0<c<1$ is small enough. Observe that $A_{t}$ is then identified to the interior of $C_{t}$. This is exposed in \cite[Cons. 4.1]{SingARR}. By \cite[Exp. 4.2]{Wolpert:hyperbolic}, we are reduced to study the continuity at 0 of the function
\begin{displaymath}
	F(t):=\int_{A_{t}}\langle\vartheta_{1}(t),\vartheta_{2}(t)\rangle_{t}\omega_{t}.
\end{displaymath}
Here $\langle\cdot,\cdot\rangle_{t}$ and $\omega_{t}$ are the metric and K\"ahler form defined by (\ref{equation:4.2a})--(\ref{equation:4.2b}). 

On $A_{t}$ we can decompose $\vartheta_{i}=(\frac{du}{u})^{k+1}(\vartheta_{iu}+\vartheta_{iv})$, $i=1,2$. The functions $\vartheta_{iu}$ (resp. $\vartheta_{iv}$) are holomorphic in the domain $|u|<c$ (resp. $|v|<c$) and vanish at $u=0$ (resp. $v=0$). Following \cite[Proof of Theorem 6, \textbf{The second term}]{Obitsu-Wolpert}, there are uniform bounds $|\vartheta_{iu}|\leq c'|u|$, $|\vartheta_{iv}|\leq c'|v|$, for some constant $c'$ depending only on $\vartheta_{1},\vartheta_{2}$. Jointly with Lemma \ref{lemma:4.2} \textit{i}, this yields
\begin{equation}\label{equation:4.3}
\begin{split}
|\int_{A_{t}}(\vartheta_{1u}\overline{\vartheta_{2v}}+\vartheta_{1v}\overline{\vartheta_{2u}})
\langle(du/u)^{k+1},(du/u)^{k+1}\rangle_{t}&\omega_{t}|=\\
&\OO(|t|(\log|t|^{-1})^{2k+1}).
\end{split}
\end{equation}
The constant involved in the $\OO$ term depends only on $\vartheta_{1}$ and $\vartheta_{2}$. The quantity (\ref{equation:4.3}) tends to $0$ as $t\to 0$. It remains to describe the contribution of the terms $\vartheta_{1u}\overline{\vartheta_{2u}}$ and $\vartheta_{1v}\overline{\vartheta_{2v}}$.

We now focus on the term corresponding to $\vartheta_{1u}\overline{\vartheta_{2u}}$. The study of $\vartheta_{1v}\overline{\vartheta_{2v}}$ is treated analogously. The forms $\vartheta_{1u}$, $\vartheta_{2u}$ depend holomorphically on $t$ and vanish at $u=0$. According to \cite[Proof of Theorem 6, \textbf{The second term}]{Obitsu-Wolpert}, the Schwarz' lemma shows that $\vartheta_{1u}(t)-\vartheta_{1u}(0)$ and $\vartheta_{2u}(t)-\vartheta_{2u}(0)$ are bounded by $c'|ut|\leq c''$. The constants $c', c''$ depend only on $\vartheta_{1}$ and $\vartheta_{2}$. The following expansion is then straightforward:
\begin{align}
	\int_{A_{t}}\vartheta_{1u}\overline{\vartheta_{2u}}\langle (du/u)^{k+1},&(du/u)^{k+1}\rangle_{t}\omega_{t}=\label{equation:4.4}\\
	&\int_{A_{t}}\vartheta_{1u}(0)\overline{\vartheta_{2u}}(0)\langle(du/u)^{k+1},(du/u)^{k+1}\rangle_{t}\omega_{t}\label{equation:4.5}\\
	&+\OO(|t|(\log|t|)^{2k}).\label{equation:4.6}
\end{align}
The constant involved in the $\OO$ term depends only on $\vartheta_{1}, \vartheta_{2}$. The quantity (\ref{equation:4.6}) converges to $0$ as $t\to 0$. To achieve the conclusion, we are thus reduced to establish the continuity of the function $G(t):=$(\ref{equation:4.5}).

Split the domain $A_{t}=\lbrace |t|/c<|u|<c\rbrace$ into the union of $A_{1t}:=\lbrace |t|^{1/2}\leq |u|<c\rbrace$ and $A_{2t}:=\lbrace |t|/c<|u|\leq |t|^{1/2}\rbrace$. Accordingly, decompose $G(t)=G_{1}(t)+G_{2}(t)$. For the first, the function $x=\pi\log|u|/\log|t|$ on $A_{1t}$ satisfies $0\leq x\leq\pi/2$. As in loc. cit., we have $0\leq x-\frac{1}{6}x^{3}\leq\sin x\leq x$ and $0\leq x^{2}-(\sin x)^{2}\leq \frac{1}{4}x^{4}-\frac{1}{36}x^{6}$. Therefore
\begin{equation}\label{equation:4.7}
	(x-(\sin x))^{2}=\OO((\log|t|)^{-4k}(\log|u|)^{6k}),\quad x=\frac{\pi\log|u|}{\log|t|},\quad u\in A_{1t}.
\end{equation}
After an elementary algebraic manipulation, the estimate (\ref{equation:4.7}) yields
\begin{equation}\label{equation:4.8}
	\begin{split}
	\langle (du)^{k+1}, &(du)^{k+1}\rangle_{t}\omega_{t}=\langle (du)^{k+1}, (du)^{k+1}\rangle_{0}\omega_{0}\\
	+&\OO((\log|t|)^{-2k}|u|^{2k}(|u|^{2k}\log|u|)^{6k}idu\wedge d\cu),\quad u\in A_{1t}.
	\end{split}
\end{equation}
Besides, observe that $|\vartheta_{1u}(0)\overline{\vartheta_{2u}}(0)||u|^{2k}=\OO(|u|^{2k+2})$. Indeed, Schwarz' lemma ensures uniform bounds $|\vartheta_{1u}|, |\vartheta_{2u}|\leq c'|u|$. Take this fact into account together with (\ref{equation:4.8}). After integration we get
\begin{equation}\label{equation:4.9}
	G_{1}(t)=\int_{A_{1t}}\vartheta_{1u}(0)\overline{\vartheta_{2u}(0)}\langle(du/u)^{k+1},(du/u)^{k+1}\rangle_{0}\omega_{0}
	+\OO((\log|t|)^{-2k}).
\end{equation}
The $\OO$ term converges to $0$ as $t\to 0$, whereas the integral on the right hand side of (\ref{equation:4.9}) is continuous --as a function of $t$-- at $t=0$. For the second function $G_{2}(t)$, there is a trivial estimate
\begin{displaymath}
	\vartheta_{1u}\overline{\vartheta_{2u}}\langle(du/u)^{k+1},(du/u)^{k+1}\rangle_{t}\omega_{t}=\OO((\log|t|)^{2k}idu\wedge d\cu),\quad u\in A_{2t}.
\end{displaymath}
After integration over $A_{2t}$, we find 
\begin{equation}\label{equation:4.10}
G_{2}(t)=\OO(|t|(\log|t|)^{2k}).
\end{equation}
Hence $G_{2}(t)\to 0$ as $t\to 0$. The proposition results from (\ref{equation:4.3})--(\ref{equation:4.6}) and (\ref{equation:4.9})--(\ref{equation:4.10}).
\end{proof}
\begin{proof}[Proof of Theorem \ref{theorem:4.3}]
For every $t\in D$, $t\neq 0$, introduce the matrix $H_{t}$ of the hermitian product $\langle\cdot,\cdot\rangle_{L^{2},t}$ on $H^{0}(\fZ_{t},\omega_{\fZ_{t}})$, in base $\theta_{i}(t)$, $\alpha_{j}(t)$, $\beta_{l}^{(k)}(t)$:
\begin{displaymath}
	H_{t}=
	\left(\begin{array}{cc}
		(\langle\theta_{i}(t),\theta_{j}(t)\rangle_{t}	)_{ij}&	\cdots\\
		\vdots	&	B_{t}
	\end{array}\right),
\end{displaymath}
where $B_{t}$ denotes the matrix
\begin{displaymath}
	\left(\begin{array}{cccc}
	(\langle\alpha_{i}(t),\alpha_{j}(t)\rangle_{t})_{ij}	&	\cdots	&	\cdots	&\cdots\\
	\vdots	&	(\langle\beta_{i}^{(1)}(t),\beta_{j}^{(1)}(t)\rangle_{t})_{i,j}	&	\cdots	&\vdots\\
	\vdots	&\vdots	&\ddots	&\vdots\\
	\vdots	&\cdots	&\cdots	&(\langle\beta_{i}^{(n)}(t),\beta_{j}^{(n)}(t)\rangle_{t})_{ij}
	\end{array}\right).
\end{displaymath}
The conclusions of propositions \ref{proposition:4.1} and \ref{proposition:4.2} guarantee that \cite[Lemma 7]{Obitsu-Wolpert} applies to $H_{t}$:
\begin{equation}\label{equation:4.11}
	\det H_{t}=\det B_{t}(\prod_{i=1}^{n}\langle\theta_{i}(t),\theta_{i}(t)\rangle_{t})(1+\OO(\delta)),
\end{equation}
with $\delta:=\sum_{i=1}^{n}\langle\theta_{i}(t),\theta_{i}(t)\rangle_{t}^{-1}$. The constant involved in the $\OO$ term does not depend on $t$. We insert the estimate of Proposition \ref{proposition:4.1} \textit{iv} into equation (\ref{equation:4.11}). By the continuity property claimed by Proposition \ref{proposition:4.2}, we find
\begin{displaymath}
	\lim_{\substack{t\to 0\\ t\neq 0}}\det H_{t}(\log|t|^{-1})^{-n(2k+1)}=
	\left(\frac{(2k)!}{2^{k-1}\pi^{2k}(k!)^{2}}\right)^{n}\det B_{0}.
\end{displaymath}
The proof of the theorem is now complete.
\end{proof}
We are ready to finish the proof of Theorem \ref{theorem:4.1}.
\begin{proof}[Proof of Theorem \ref{theorem:4.1}]
The theorem is an immediate application of Corollary \ref{corollary:4.1}, Theorem \ref{theorem:4.3} and the definition of the Quillen metric (Definition \ref{definition:Quillen}).
\end{proof}
\section{Metrized Mumford isomorphisms on $\SM_{g,n}$}\label{section:metrized}
This section focuses on an \textit{arithmetic} counterpart of Theorem \ref{theorem:3.1}. We derive some consequences, such as the Takhtajan-Zograf local index theorem \cite{ZT1}--\cite{ZT2} and theorems A and B.

Recall from Section \ref{section:conventions} that the invertible sheaves $\lambda_{k+1;g,n}$ and $\psi_{g,n}$ on $\SM_{g,n}$ come equipped with the Quillen and Wolpert metrics, respectively. The trivial sheaf together with the norm $C(g,n)|\cdot|$ is denoted by $\OO(C(g,n))$. Here $C(g,n)$ is the constant introduced in Section \ref{section:conventions} and $|\cdot|$ is the absolute value. Finally, $\kappa_{g,n}$ is enriched with the Liouville metric \cite[Sec. 4]{SingARR}. As a special feature, the Liouville metric is defined and \textit{continuous} over $\SCM_{g,n}$. Moreover, if $\gamma:\SCM_{g,n}\times\SCM_{1,1}^{\times n}\rightarrow\SCM_{g+n,0}$ is a reiterated clutching morphism, then the isomorphism (\ref{equation:3.7}) $\gamma^{\ast}\kappa_{g+n,0}\overset{\sim}{\rightarrow}\kappa_{g,n}\boxtimes\kappa_{1,1}^{\boxtimes n}$ becomes an isometry for the Liouville metrics. The reader is referred to loc. cit. for full details.
\begin{theorem}\label{theorem:5.1}
Let $k\geq 0$ be an integer and $\DD^{\circ}_{k+1;g,n}$ the restriction of the Mumford isomorphism $\DD_{k+1;g,n}$ to $\SM_{g,n}$. Then $\DD^{\circ}_{k+1;g,n}$ induces an isometry
\begin{displaymath}
	\overline{\DD}^{\circ}_{k+1;g,n}:\lambda_{k+1;g,n;Q}^{\otimes 12}\otimes\psi_{g,n;W}
	\overset{\sim}{\longrightarrow}\overline{\kappa}_{g,n}^{\otimes (6k^{2}+6k+1)}\otimes\OO(C(g,n)).
\end{displaymath}
\end{theorem}
The proof of the theorem follows the same pattern as for \cite[Th. 6.1]{SingARR}. The first step consists in a simple remark: the case $g\geq 2$, $n=0$ is already known from the work of Deligne \cite{Deligne} and Gillet-Soul\'e \cite{ARR}.
\begin{proposition}\label{proposition:5.1}
Let $k\geq 1$, $g\geq 2$ be integers. Then:\\
i. Theorem \ref{theorem:5.1} holds true for $(g,n)=(g,0)$;\\
ii. endow $\delta_{g,0}^{-1}=\OO(-\pd\SM_{g,0})$ with the trivial singular metric coming from the absolute value; write $\overline{\delta}_{g,0}^{-1}$ for the resulting hermitian line bundle. Then $\lambda_{k+1;g,0;Q}^{\otimes 12}$ extends to a continuous hermitian line bundle $\lambda_{k+1;g,0;Q}^{\otimes 12}\otimes\overline{\delta}_{g,0}^{-1}$ on $\SCM_{g,0}$. Moreover, $\overline{\DD}^{\circ}_{k+1;g,0}$ extends to an isometry of continuous hermitian line bundles on $\SCM_{g,n}$
\begin{displaymath}
	\lambda_{k+1;g,0;Q}^{\otimes 12}\otimes\overline{\delta}_{g,0}^{-1}\overset{\sim}{\longrightarrow}
	\overline{\kappa}_{g,0}^{\otimes (6k^{2}+6k+1)}\otimes\OO(C(g,0)).
\end{displaymath}
\end{proposition}
\begin{proof}
The first assertion is due to Deligne \cite[Th. 11.4]{Deligne} and Gillet-Soul\'e \cite{ARR}: it suffices to check that our definition of Quillen metric agrees with theirs. Let $X$ be a compact Riemann surface of genus $g\geq 2$. Endow the holomorphic tangent bundle $T_{X}$ with the hermitian metric attached to the riemannian hyperbolic metric of constant curvature $-1$ on $X$. The canonical sheaf $\omega_{X}$ is equipped with the dual metric. With these choices, our $L^{2}$ metric on $H^{0}(X,\omega_{X}^{k+1})$ coincides with the one of Deligne and Gillet-Soul\'e. We still need to see that our normalization of the $L^{2}$ metric defining the Quillen metric (Definition \ref{definition:Quillen}) is given by $\exp(\frac{1}{2}T(\overline{\omega}_{X}^{k+1}))$. Here $T(\overline{\omega}_{X}^{k+1})$ is the analytic torsion of $\overline{\omega}_{X}^{k+1}$, for our choice of K\"ahler metric on $T_{X}$ (Section \ref{section:conventions}). The evaluation of $T(\overline{\omega}_{X}^{k+1})$ has been indicated by Sarnak \cite[p. 116, par. 1]{Sarnak} and completed --for instance-- in \cite[Cor. 1.12]{Ebel}:
\begin{displaymath}
		T(\overline{\omega}_{X}^{k+1})=-\log Z(X, k+1) -(2g-2)C_{k},
\end{displaymath}
with $C_{k}$ being the constant
\begin{displaymath}
	\begin{split}
		C_{k}=&2\zeta^{\prime}(-1)-\left(k+\frac{1}{2}\right)^{2}+\left(k+\frac{1}{2}\right)\log(2\pi)\\
		+&\sum_{j=1}^{2k}\left(j-k-\frac{1}{2}\right)\log j\\
		+&\left(\frac{k+1}{2}-\frac{1}{3}\right)\log 2.
	\end{split}
\end{displaymath}
Taking into account the definition of $E_{k+1}(g,0)$, this proves the first claim. The second item already follows from \textit{i}, Theorem \ref{theorem:3.1} and the continuity of the Liouville metric on $\kappa_{g,0}$ on $\SCM_{g,0}$ \cite[Th. 4.7]{SingARR}.
\end{proof}
\begin{proof}[Proof of Theorem \ref{theorem:5.1}]
The case $k=0$ is the content of \cite[Th. 6.1]{SingARR}. Now for $k\geq 1$. Due to Corollary \ref{corollary:3.1}, Corollary \ref{corollary:3.2}, Lemma \ref{lemma:4.1}, Theorem \ref{theorem:4.1} and Proposition \ref{proposition:5.1}, the proof is formally the same as for $k=0$. We refer to \cite[Sec. 6]{SingARR} for the details.
\end{proof}
\begin{corollary}
Let $k\geq 1$ be an integer.\\
i. The hermitian line bundle $\lambda_{k+1;g,n;Q}\otimes\lambda_{k;g,n;Q}^{-1}$ on $\SM_{g,n}$ extends to a continuous hermitian line bundle on $\SCM_{g,n}$, whose underlying sheaf is $\lambda_{k+1;g,n}\otimes\lambda_{k;g,n}^{-1}$.\\
ii. There is an isometry of continuous hermitian line bundles on $\SCM_{g,n}$
\begin{displaymath}
	\lambda_{k+1;g,n;Q}\otimes\lambda_{k;g,n;Q}^{-1}\overset{\sim}{\longrightarrow}\overline{\kappa}_{g,n}^{\otimes k}.
\end{displaymath}
\end{corollary}
\begin{proof}
First of all, from the isomorphism (\ref{equation:3.8}) of Lemma \ref{lemma:3.1} we get an isomorphism
\begin{equation}\label{equation:5.0.1}
	\DD:\lambda_{k+1;g,n}\otimes\lambda_{k;g,n}^{-1}\overset{\sim}{\longrightarrow}\kappa_{1}^{\otimes k}.
\end{equation}
By \cite[Cor.3.2]{SingARR} and Theorem \ref{theorem:5.1}, $\DD^{\otimes 12}$ coincides, up to a sign, with $\DD_{k+1;g,n}\otimes\DD_{k;g,n}^{\otimes -1}$. Therefore, $\DD\mid_{\SM_{g,n}}^{\otimes 12}$ induces an isometry
\begin{equation}\label{equation:5.0.2}
	\overline{\DD}\mid_{\SM_{g,n}}^{\otimes 12}:(\lambda_{k+1;g,n;Q}\otimes\lambda_{k;g,n;Q}^{-1})^{\otimes 12}\overset{\sim}{\longrightarrow}\overline{\kappa}_{1}^{\otimes 12k}.
\end{equation}
We conclude by (\ref{equation:5.0.1})--(\ref{equation:5.0.2}) and the continuity of the Liovulle metric on $\SCM_{g,n}$ \cite[Cor. 4.8]{SingARR}.
\end{proof}
\begin{remark}
Observe that the metric on $\lambda_{k+1;g,n;Q}\otimes\lambda_{k;g,n;Q}^{-1}$, evaluated for a smooth  $n$-pointed stable curve $(X;p_{1},\ldots,p_{n})$ over $\CC$, involves the factor $R^{\prime}(U,1)$ if $k=0$ and $R(U,k+1)$ if $k\geq 1$, where $R(U,s)$ is the Ruelle z\^eta function of $U=X\setminus\lbrace p_{1},\ldots, p_{n}\rbrace$.\footnote{Recall that $R(U,s)=Z(U,s)/Z(U,s+1)$, for $Z$ the Selberg z\^eta function of $U$.}
\end{remark}
\begin{corollary}[Takhtajan-Zograf local index formula \cite{ZT1}--\cite{ZT2}]\label{corollary:5.2}
Let $\omega_{WP}$, $\omega_{TZ}$ be the Weil-Petersson and Takhtajan-Zograf K\"ahler forms on $\SM_{g,n}^{\an}$, respectively. For every integer $k\geq 0$, the following equality of differential forms on $\SM_{g,n}^{\an}$ holds:
\begin{equation}\label{equation:5.3}
	\c1(\lambda_{k+1;g,n;Q})=\frac{6k^{2}+6k+1}{12\pi^{2}}\omega_{WP}-\frac{1}{9}\omega_{TZ}.
\end{equation}
\end{corollary}
\begin{proof}
Firstly, there is an equality of differential forms on $\SM_{g,n}^{\an}$
\begin{equation}\label{equation:5.4}
	\c1(\overline{\kappa}_{g,n})=\frac{1}{\pi^{2}}\omega_{WP}.
\end{equation}
For a proof we refer to \cite{Wolpert:hyperbolic} (case $n=0$) and \cite[Ch. 5]{GFM:thesis} (general case). Secondly, \cite[Th. 5]{Wolpert:cusps} proves the identity
\begin{equation}\label{equation:5.5}
	\c1(\psi_{g,n;W})=\frac{4}{3}\omega_{TZ}.
\end{equation}
The relation (\ref{equation:5.3}) is obtained from Theorem \ref{theorem:5.1} and (\ref{equation:5.4})--(\ref{equation:5.5}).
\end{proof}
\begin{proof}[Proof of Theorem A]
The theorem is a consequence of Theorem \ref{theorem:3.1} and Theorem \ref{theorem:5.1}. The argument is the same as for \cite[Th. A]{SingARR}.
\end{proof}
\begin{proof}[Proof of Theorem B]
Let $\UU=\XX\setminus\cup_{j}\sigma_{j}(\BS)$. By Theorem A \textit{ii}, it suffices to show that for every complex embedding $\sigma:K\hookrightarrow\CC$,
\begin{displaymath}
	\log (E_{k+1}(g,n)Z(\UU_{\sigma}(\CC),k+1))=\OO(k\log k).
\end{displaymath}
On the one hand, $\log Z(\UU_{\sigma}(\CC),k+1)=o(1)$. We are left to prove that $\log E_{k+1}(g,n)=\OO(k\log k)$. From the definition of $E_{k+1}(g,n)$, the relevant terms to treat are:
\begin{itemize}
	\item[\textit{i}.] terms in $\log(2k)!$; these are $\OO(k\log k)$ by Stirling's formula;
	\item[\textit{ii}.] $\alpha_{k}:=(\sum_{j=1}^{2k}(j-k)\log j)-k^{2}$; we can write
	\begin{displaymath}
		\alpha_{k}=\log\Gamma_{2}(2k+1)+k\log(2k)!-k^{2},
	\end{displaymath}
	where $\Gamma_{2}$ is the Barnes' double gamma function \cite{Barnes} (see also \cite{Sarnak}). By loc. cit. and Stirling's formula we have
	\begin{align}
		&\log\Gamma_{2}(2k+1)=(2k)^{2}\left(\frac{\log(2k)}{2}-\frac{3}{4}\right)+\OO(k)\label{equation:5.1}\\
		&k\log(2k)!=2k^{2}\log(2k)-3k^{2}+\OO(k\log k).\label{equation:5.2}
	\end{align}
	From the expansions (\ref{equation:5.1})--(\ref{equation:5.2}) we infer $\alpha_{k}=\OO(k\log k)$.
\end{itemize}
The proof is complete.
\end{proof}
\section{Application to pointed stable curves of ge\-nus 0}\label{section:applications}
Let $K$ be a number field, whose ring of integers we denote $\OO_{K}$. Set $\BS=\Spec\OO_{K}$. Consider a pointed stable curve $(\pi:\XX\rightarrow\BS;\sigma_{1},\ldots,\sigma_{n})$ and define $\UU=\XX\setminus\cup_{j}\sigma_{j}(\BS)$. Assume that $\pi$ is generically smooth. Recall the definition of the $L^{2}$ metric of Section \ref{section:conventions}. For every integer $k\geq 0$ and every field embedding $\sigma:K\hookrightarrow\CC$, the complex vector space $H^{0}(\XX,\omega_{\XX/\BS}^{k+1}(k\sigma_{1}+\ldots+k\sigma_{n}))\otimes_{\sigma}\CC$ comes equipped with the $L^{2}$ metric attached to the hyperbolic metric on $\UU_{\sigma}(\CC)$. Let it be $\langle\cdot,\cdot\rangle_{L^{2},\sigma}$. On the other hand, by the uniformization theorem, there is a biholomorphism $\UU_{\sigma}(\CC)\simeq\Gamma_{\sigma}\backslash\HH$, for some discrete torsion subgroup $\Gamma_{\sigma}\subset\PSL_{2}(\RR)$. The space of modular parabolic forms of weight $2k+2$ for $\Gamma_{\sigma}$, $S_{2k+2}(\Gamma_{\sigma})$, is canonically isomorphic to $H^{0}(\XX,\omega_{\XX/\BS}^{k+1}(k\sigma_{1}+\ldots+\sigma_{n}))\otimes_{\sigma}\CC$. The Petersson metric on $S_{2k+2}(\Gamma_{\sigma})$ induces a hermitian metric on $H^{0}(\XX,\omega_{\XX/\BS}^{k+1}(k\sigma_{1}+\ldots+k\sigma_{n}))\otimes_{\sigma}\CC$, to be denoted $\langle\cdot,\cdot\rangle_{\Pet,\sigma}$. It is readily checked that $\langle\cdot,\cdot\rangle_{L^{2},\sigma}=(2^{k}/\pi)\langle\cdot,\cdot\rangle_{\Pet,\sigma}$. This relation stems from the comparison between the hermitian and riemannian hyperbolic metrics, as well as the normalization of the K\"ahler form $\omega$ (Section \ref{section:conventions}). 

The space of global sections $H^{0}(\omega_{\XX/\BS}^{k+1}(k\sigma_{1}+\ldots+k\sigma_{n}))$ is a projective $\OO_{K}$-module. As a $\Int$-module, it becomes a lattice in the real vector space
\begin{displaymath}
	\begin{split}
	H^{0}(\XX&,\omega_{\XX/\BS}^{k+1}(k\sigma_{1}+\ldots+k\sigma_{n}))\otimes_{\Int}\CC=\\
		&\left(\bigoplus_{\sigma:K\hookrightarrow\CC}
	H^{0}(\XX,\omega_{\XX/\BS}^{k+1}(k\sigma_{1}+\ldots+k\sigma_{n}))\otimes_{\sigma}\CC\right)^{F_{\infty}}.
	\end{split}
\end{displaymath}
This vector space is endowed with the norms
\begin{displaymath}
	\|f\|_{L^{2},\infty}=\sup_{\sigma:K\hookrightarrow\CC}\|f_{\sigma}\|_{L^{2},\sigma}\quad\text{and}\quad\|f\|_{\Pet,\infty}=\sup_{\sigma:K\hookrightarrow\CC}\|f_{\sigma}\|_{\Pet,\sigma}.
\end{displaymath}
\begin{proposition}
Assume $(\omega_{\XX/\BS}(\sigma_{1}+\ldots+\sigma_{n})_{\hyp}^{2})>0$. Then there exists an integer $k_{0}\geq 0$ such that
\begin{displaymath}
	\lbrace f\in H^{0}(\omega_{\XX/\BS}^{k+1}(k\sigma_{1}+\ldots+k\sigma_{n}))\mid \| f\|_{\Pet,\infty}^{2}\leq\frac{\pi}{2^{k}}\rbrace\neq\emptyset\quad\text{for all}\quad k\geq k_{0}.
\end{displaymath}
\end{proposition}
\begin{proof}
By Theorem B and Minkowski's theorem, we infer that there exists $k_{0}\geq 0$ such that
\begin{displaymath}
	\lbrace f\in H^{0}(\omega_{\XX/\BS}^{k+1}(k\sigma_{1}+\ldots+k\sigma_{n}))\mid 0<\| f\|_{L^{2},\infty}\leq 1\rbrace\neq\emptyset\quad\text{for all}\quad k\geq k_{0}.
\end{displaymath}
See \cite[Sec. 5.2]{ARR}, \cite[Ch. VIII, Par. 2.3]{Soule} or \cite[Exp. III, Th. 4]{Szpiro} for the details of the argument. We conclude by the relation $\| f\|_{\Pet,\infty}^{2}=(\pi/2^{k})\|f\|_{L^{2},\infty}^{2}$.
\end{proof}
Fix an integer $n\geq 3$. If $p_{1},\ldots,p_{n}\in\PP_{K}^{1}(K)$ are distinct $K$-valued points, then $(\PP^{1}_{K};p_{1},\ldots,p_{n})$ is a pointed stable curve of genus $0$ over $K$. The moduli space $\SCM_{0,n}$ is a smooth and projective scheme over $\Spec\Int$. Therefore the tuple $(\PP^{1}_{K};p_{1},\ldots,p_{n})$ extends to a pointed stable curve $(\pi:\XX\rightarrow\BS;\wp_{1},\ldots,\wp_{n})$, with a cartesian diagram
\begin{displaymath}
	\xymatrix{
		\XX\ar[r]\ar[d]^{\pi}	&\SCM_{0,n+1}\ar[d]\\
		\BS\ar[r]^{\hspace{-0.3cm}\C(\pi)}\ar@/^/[u]^{\wp_{1},\ldots,\wp_{n}}		&\SCM_{0,n}\ar@/_/[u]_{\sigma_{1},\ldots,\sigma_{n}}.
	}
\end{displaymath}
The real number $(\omega_{\XX/\BS}(\sigma_{1}+\ldots+\sigma_{n})_{\hyp}^{2})=\adeg(\C(\pi)^{\ast}\overline{\kappa}_{0,n})$ is a well-defined invariant of $(\PP^{1}_{K};p_{1},\ldots,p_{n})$. From \cite[Cor. 4.8]{SingARR}, we know that the Liouville metric on $\kappa_{0,n}$ is continuous on $\SCM_{0,n}$.\footnote{In \cite[Ch. 5]{GFM:thesis} we prove that the Liouville metric on $\kappa_{g,n}$ is continuous on $\SCM_{g,n}$ and pre-log-log along $\pd\SM_{g,n}$.} Accordingly, there is a well defined arakelovian height on $1$-cycles\footnote{Recall that a $1$-cycle of $\SCM_{0,n}$ is a formal linear combination, with integer coefficients, of integral closed Zariski subsets of $\SCM_{0,n}$.} on $\SCM_{0,n}$ with respect to $\overline{\kappa}_{0,n}$, $h_{\overline{\kappa}_{0,n}}$. We thus rewrite $\adeg(\C(\pi)^{\ast}\overline{\kappa}_{0,n})=h_{\overline{\kappa}_{0,n}}(\Imag\C(\pi))$, where $\Imag\C(\pi)$ is the image cycle of $\C(\pi)$ in $\SCM_{0,n}$. 
\begin{theorem}
Let $K$ be a fixed number field and $n\geq 3$ an integer. For all but a finite number of $n$-tuples $(p_{1},\ldots,p_{n})$ of distinct $K$-rational points of $\PP^{1}_{K}$, there exists an integer $k_{0}=k_{0}(p_{1},\ldots, p_{n})$ such that
\begin{displaymath}
	\lbrace f\in H^{0}(\XX,\omega_{\XX/\BS}^{k+1}(k\wp_{1}+\ldots+k\wp_{n})\mid 0<\|f\|_{\Pet,\infty}^{2}\leq\frac{\pi}{2^{k}}\rbrace\neq\emptyset
	\quad\text{for all}\quad k\geq k_{0}.
\end{displaymath}
\end{theorem}
\begin{proof}
Notations being as above, we have to prove that there are finitely many tuples of distinct $K$-rational points $(p_{1},\ldots,p_{n})$ with $h_{\overline{\kappa}_{0,n}}(\Imag\C(\pi))\leq 0$. Because $K$ has been fixed and the Liouville metric is continuous, it suffices to show that $\kappa_{0,n}$ is ample on $\SCM_{0,n}$. This is precisely the content of Theorem \ref{theorem:ampleness} in the appendix. The proof is complete.
\end{proof}
\section{Appendix}
Let $n\geq 3$ be an integer and $\SCM_{0,n}\rightarrow\Spec\Int$ the moduli scheme classifying $n$-pointed stable curves of genus $0$. We identify $\SCM_{0,n+1}\rightarrow\SCM_{0,n}$ with the universal curve and write $\sigma_{1},\ldots,\sigma_{n}$ for the universal sections. In this appendix we establish the ampleness of the tautological line bundle $\kappa_{0,n}=\langle\omega_{\SCM_{0,n+1}/\SCM_{0,n}}(\sigma_{1}+\ldots+\sigma_{n}),\omega_{\SCM_{0,n+1}/\SCM_{0,n}}(\sigma_{1}+\ldots+\sigma_{n})\rangle$.
\begin{lemma}[Keel-Tevelev]\label{theorem:keel-tevelev}
Let $k$ be any field.\\
i. The line bundle $\kappa_{0,n;k}$ on $\SCM_{0,n;k}$ is very ample.\\
ii. For every integer $m\geq 1$, let $W_{m}$ be the standard irreducible representation of the symmetric group $\mathfrak{S}_{m}$, i.e. $m$-tuples of integers that sum to 0. Then
\begin{displaymath}
	H^{0}(\SCM_{0,n;k},\kappa_{0,n;k})\simeq\bigotimes_{4\leq j\leq n}W_{j}.
\end{displaymath}
In particular the dimension $\dim H^{0}(\SCM_{0,n;k},\kappa_{0,n;k})$ does not depend on the field $k$.
\end{lemma}
\begin{proof}
This is Corollary 2.6 and Corollary 2.7 of \cite{Keel-Tevelev} (observe that in Section 2 of loc. cit. $k$ is an arbitrary field).
\end{proof}
\begin{theorem}\label{theorem:ampleness}
The line bundle $\kappa_{0,n}$ on $\SCM_{0,n}$ is ample.
\end{theorem}
\begin{proof}
We check that $\kappa_{0,n}$ satisfies the cohomological criterion of ampleness \cite[Ch. III, Prop. 5.3]{Hartshorne}. Let $\mathcal{F}$ be a coherent sheaf on $\SCM_{0,n}$. We must prove that there exists an integer $N_{0}>0$ such that
\begin{equation}\label{equation:app_1}
	H^{i}(\SCM_{0,n},\mathcal{F}\otimes\kappa_{0,n}^{\otimes N})=0\quad\text{for}\quad N\geq N_{0}\quad\text{and}\quad i>0.
\end{equation}
A first reduction consists in restricting to the class of coherent sheaves which are flat over $\Spec\Int$. Indeed, suppose we have checked the criterion of ampleness for such sheaves. Since $\SCM_{0,n}$ is projective over $\Spec\Int$, for any coherent sheaf $\mathcal{F}$ on $\SCM_{0,n}$ there is an exact sequence of coherent sheaves
\begin{equation}\label{equation:app_2}
	0\longrightarrow \mathcal{K}\longrightarrow\mathcal{E}\longrightarrow\mathcal{F}\longrightarrow 0,
\end{equation}
with $\mathcal{E}$ locally free. In particular, $\mathcal{E}$ is flat over $\SCM_{0,n}$. Since $\SCM_{0,n}$ is flat over $\Spec\Int$, so does $\mathcal{E}$. Because $\Int$ is a domain of principal ideals, flatness of $\mathcal{E}$ over $\Spec\Int$ implies flatness of $\mathcal{K}$ over $\Spec\Int$. Fix an integer $N_{0}$ such that (\ref{equation:app_1}) holds for both $\mathcal{K}$ and $\mathcal{E}$. Then (\ref{equation:app_1}) also holds for $\mathcal{F}$, by the long exact sequence of cohomology of (\ref{equation:app_2}). Henceforth $\mathcal{F}$ denotes a coherent sheaf flat over $\Spec\Int$.

By Lemma \ref{theorem:keel-tevelev} the line bundle $\kappa_{0,n;\QQ}$ is very ample on $\SCM_{0,n;\QQ}$. If $\pi:\SCM_{0,n}\rightarrow\Spec\Int$ stands for the structure map, then  the natural morphism
\begin{equation}\label{equation:app_3}
	\pi^{*}\pi_{*}\kappa_{0,n}\longrightarrow\kappa_{0,n}
\end{equation}
becomes surjective when restricted to $\SCM_{0,n;\QQ}$, by global generation of $\kappa_{0,n;\QQ}$. Therefore, for a sufficiently divisible integer $l$, (\ref{equation:app_3}) becomes surjective when restricted to $\SCM_{0,n;\Int[1/l]}$. This yields a commutative diagram of proper morphisms
\begin{equation}\label{equation:app_4}
	\xymatrix{
		\SCM_{0,n;\Int[1/l]}\ar[r]^{\hspace{-0.1cm}\varphi}\ar[rd]
		&\PP(\pi_{*}\kappa_{0,n;\Int[1/l]})\ar[d]\\
		&\Spec\Int[1/l].
	}
\end{equation}
From Lemma \ref{theorem:keel-tevelev} the dimension of $H^{0}(\SCM_{0,n;\FF_{p}},\kappa_{0,n;\FF_{p}})$ is independent of $p$. By Grauert's theorem \cite[Ch. III, Cor. 12.9]{Hartshorne}, the sheaf $\pi_{*}\kappa_{0,n;\Int[1/l]}$ is locally free on $\Spec\Int[1/l]$. In addition, for every $s\in\Spec\Int[1/l]$ we have an isomorphism
\begin{equation}\label{equation:app_5}
	\pi_{*}\kappa_{0,n;\Int[1/l]}\otimes k(s)\overset{\sim}{\longrightarrow} H^{0}(\SCM_{0,n;k(s)},\kappa_{0,n;k(s)}).
\end{equation}
If $p$ is a prime number with $p\nmid l$, the isomorphism (\ref{equation:app_5}) amounts to
\begin{displaymath}
	\pi_{*}\kappa_{0,n;\Int[1/l]}\otimes_{\Int}\FF_{p}\overset{\sim}{\longrightarrow}H^{0}(\SCM_{0,n;\FF_{p}},\kappa_{0,n;\FF_{p}}).
\end{displaymath}
Therefore, reducing (\ref{equation:app_4}) modulo $p$, with $p\nmid l$, yields
\begin{displaymath}
	\SCM_{0,n;\FF_{p}}\overset{\varphi_{p}}{\longrightarrow}\PP(H^{0}(\SCM_{0,n;\FF_{p}},\kappa_{0,n;\FF_{p}})).
\end{displaymath}
The morphism $\varphi_{p}$ is a closed immersion, since $\kappa_{0,n;\FF_{p}}$ is very ample. This being true for all $p\nmid l$, the morphism $\varphi$ of (\ref{equation:app_4}) is also a closed immersion. We infer that $\kappa_{0,n;\Int[1/l]}$ is very ample. Hence, there exists an integer $N_{0}>0$ such that
\begin{equation}\label{equation:app_6}
	H^{i}(\SCM_{0,n;\Int[1/l]},\mathcal{F}\otimes\kappa^{\otimes N}_{0,n;\Int[1/l]})=0\quad\text{for}\quad N\geq N_{0}\quad\text{and}\quad i>0.
\end{equation}
Since higher direct images commute with flat base change \cite[Ch. III, Prop. 9.3]{Hartshorne}, (\ref{equation:app_6}) equivalently reads
\begin{equation}\label{equation:app_7}
	R^{i}\pi_{*}(\mathcal{F}\otimes\kappa_{0,n}^{\otimes N})(\Spec\Int[1/l])=0\quad\text{for}\quad N\geq N_{0}\quad\text{and}\quad i>0.
\end{equation}
We still need to deal with primes $p\mid l$. Let $p$ be such a prime. Since $\kappa_{0,n;\FF_{p}}$ is very ample, there exists $N_{p}>0$ for which the ampleness criterion holds for $\kappa_{0,n;\FF_p}$ and $\mathcal{F}_{\FF_p}$. Since there are only finitely many primes dividing $l$, after possibly increasing $N_{0}$ we can suppose that $N_{0}\geq N_{p}$ for all $p\mid l$. Let us fix an integer $N\geq N_{0}$. For every $p\mid l$ and $i=1,\ldots, n-2$, introduce
\begin{displaymath}
	\Omega_{p,i}=\lbrace y\in\Spec\Int\mid H^{i}(\SCM_{0,n;k(y)},\mathcal{F}\otimes\kappa^{\otimes N}_{0,n;k(y)})=0\rbrace.
\end{displaymath}
Then $\Omega_{p,i}$ is Zariski open by the semi-continuity theorem \cite[Ch. III, Th. 12.8]{Hartshorne}, and contains $p$ because $N\geq N_0\geq N_{p}$. We write $\Omega_p$ for the intersection of all the $\Omega_{p,i}$ for $i=1,\ldots,n-2$. Define
\begin{displaymath}
 	\Omega_{p}=\bigcap_{i=0}^{n-2}\Omega_{p,i},
\end{displaymath}
which is a Zariski open neighborhood of $p$. Let $l_{p}\geq 1$ be an integer such that $p\in\Spec\Int[1/l_p]\subseteq\Omega_{p}$. The construction of $\Omega_{p}$ and the semi-continuity theorem ensure
\begin{equation}\label{equation:app_8}
	R^{i}\pi_{*}(\mathcal{F}\otimes\kappa_{0,n}^{\otimes N})(\Spec\Int[1/l_p])=0\quad\text{for}\quad i=1,\ldots,n-2.
\end{equation}
Finally, observe that $\mathcal{U}:=\lbrace\Spec\Int[1/l]\rbrace\cup\lbrace\Spec\Int[1/l_p]\rbrace_{p\mid l}$ is an affine open covering of $\Spec\Int$. From equations (\ref{equation:app_7})--(\ref{equation:app_8}), it comes
\begin{displaymath}
	H^{i}(\SCM_{0,n},\mathcal{F}\otimes\kappa_{0,n}^{\otimes N})=R^{i}\pi_{*}(\mathcal{F}\otimes\kappa_{0,n}^{\otimes N})(\Spec\Int)=0
	\quad\text{for}\quad i=1,\ldots,n-1.
\end{displaymath}
This even holds for $i>n-2$, because $\SCM_{0,n}$ has Krull dimension $n-2$. The proof is complete.
\end{proof}
\bibliographystyle{amsplain}

\textsc{G. Freixas i Montplet, D\'epartement de Math\'ematiques, Universit\'e Paris-Sud 11, B\^atiment 425, 91405 Orsay cedex, France}\\
\\
\textit{E-mail address}: \texttt{gerard.freixas@math.u-psud.fr}
\end{document}